\documentclass[a4paper,11pt]{amsart}

\usepackage{a4wide,amsmath,amsfonts,amssymb,amsthm,bbm}
\usepackage[hyperfootnotes=false]{hyperref}

\allowdisplaybreaks[2]

\numberwithin{equation}{section}

\newtheorem{theorem}{Theorem}[section]
\newtheorem{lemma}[theorem]{Lemma}
\newtheorem{corollary}[theorem]{Corollary}
\newtheorem{remark}[theorem]{Remark}
\newtheorem{proposition}[theorem]{Proposition}
\newtheorem{assumption}[theorem]{Assumption}

\newcommand{\dd}{\,\mathrm{d}}
\newcommand{\R}{\mathbb{R}}
\newcommand{\N}{\mathbb{N}}
\newcommand{\E}{\mathbb{E}}
\newcommand{\1}{\mathbf{1}}

\title[Stochastic Volterra equations with H{\"o}lder diffusion coefficients]{Stochastic Volterra equations \\ with H{\"o}lder diffusion coefficients}

\author[Pr{\"o}mel]{David J. Pr{\"o}mel}
\address{David J. Pr{\"o}mel, University of Mannheim, Germany}
\email{proemel@uni-mannheim.de}

\author[Scheffels]{David Scheffels}
\address{David Scheffels, University of Mannheim, Germany}
\email{dscheffe@mail.uni-mannheim.de}

\date{\today}

\begin{document}
 
\begin{abstract}
  The existence of strong solutions and pathwise uniqueness are established for one-dimensional stochastic Volterra equations with locally H{\"o}lder continuous diffusion coefficients and sufficiently regular kernels. Moreover, we study the sample path regularity, the integrability and the semimartingale property of solutions to one-dimensional stochastic Volterra equations.
\end{abstract}

\maketitle





\section{Introduction}

Stochastic Volterra equations (SVEs) have been studied in probability theory starting with the works of Berger and Mizel~\cite{Berger1980a,Berger1980b}. This class of integral equations constitutes a generalization of ordinary stochastic differential equations and serves as well suited mathematical model for numerous random phenomena appearing, e.g., in biology, physics and mathematical finance.

In the present work, we investigate the strong existence and pathwise uniqueness of solutions to one-dimensional stochastic Volterra equations with locally H{\"o}lder continuous diffusion coefficients and sufficiently regular kernels. More precisely, we consider SVEs of the form
\begin{equation}\label{eq:intro}
  X_t = x_0(t)+\int_0^t K_{\mu}(s,t)\mu(s,X_s)\dd s+\int_0^t K_{\sigma}(s,t)\sigma(s,X_s)\dd B_s,\quad t\in [0,T],
\end{equation}
where $x_0$ denotes the initial condition, $(B_t)_{t\in [0,T]}$ is a Brownian motion, the kernels $K_\mu, K_\sigma$ are sufficiently regular functions, the coefficient~$\mu$ is locally Lipschitz continuous, and the diffusion coefficient~$\sigma$ is locally H{\"o}lder continuous.

The motivation to study stochastic Volterra equations with non-Lipschitz coefficients is twofold. On the one hand, it is a natural question to explore to what extent the famous results of Yamada and Watanabe~\cite{Yamada1971}, ensuring pathwise uniqueness and the existence of strong solutions for ordinary stochastic differential equations, generalizes to stochastic Volterra equations. On the other hand, stochastic Volterra equations with only $1/2$-H{\"o}lder continuous coefficients recently got a great deal of attention in mathematical finance as so-called rough volatility models, see e.g. \cite{AbiJaberElEuch2019b,ElEuch2019}, which have demonstrated to fit remarkably well historical and implied volatilities of financial markets, see e.g. \cite{Bayer2016}. Furthermore, SVEs with non-Lipschitz continuous coefficients arise as scaling limits of branching processes in population genetics, see \cite{Mytnik2015,AbiJaber2021b}.

The existence of unique strong solutions for stochastic Volterra equations with Lipschitz continuous coefficients is well investigated. Indeed, classical existence and uniqueness results for SVEs with sufficiently regular kernels are due to \cite{Berger1980a,Berger1980b,Protter1985}. These results have been generalized in various directions such as allowing for anticipating and path-dependent coefficients \cite{Pardoux1990,Oksendal1993,Alos1997,Kalinin2021}, singular kernels \cite{Cochran1995,Coutin2001} or an infinite dimensional setting \cite{Zhang2010}. A slight extension beyond Lipschitz continuous coefficients can be found in \cite{Wang2008}.

The classical approach to prove the existence of strong solutions to ordinary stochastic differential equations with less regular diffusion coefficients is to first show the existence of a weak solution, since this, in combination with pathwise uniqueness, guarantees the existence of a strong solution, see~\cite{Yamada1971}. Only recently, the existence of weak solutions for stochastic Volterra equations was derived in the work of Abi Jaber, Cuchiero, Larsson and Pulido~\cite{AbiJaber2021} (see also \cite{Mytnik2015,AbiJaber2019,AbiJaber2021b}), assuming that the kernels in the stochastic Volterra equations are of convolution type, i.e. in our setting $K_\mu (s,t)= K_\sigma(s,t)=K(t-s)$ for some function $K\colon \R\to \R$. Assuming additionally that the coefficients $\mu,\sigma$ lead to affine Volterra processes, weak uniqueness was obtained in \cite{Mytnik2015,AbiJaberElEuch2019,AbiJaber2021b,Cuchiero2020}. However, as we do not impose a convolution structure on the stochastic Volterra equation~\eqref{eq:intro}, we cannot rely on the known results regarding the existence of weak solutions.

Our first main contribution is to establish the existence of a strong solution to the SVE~\eqref{eq:intro} provided the diffusion coefficient~$\sigma$ is locally $1/2+\xi$-H{\"o}lder continuous for $\xi \in [0,1/2]$. To that end, we prove the convergence of an Euler type approximation of the SVE~\eqref{eq:intro} and do not use the concept of weak solutions. For ordinary stochastic differential equations such an approach was developed by Gy{\"o}ngy and R{\'a}sonyi~\cite{Gyongy2011}, using ideas coming from~\cite{Yamada1971}. As a number of results used to deal with ordinary stochastic differential equations are not available in the context of SVEs, the presented proof for the existence of a strong solution to the SVE~\eqref{eq:intro} requires various different techniques such as a transformation formula for Volterra processes {\`a} la Protter~\cite{Protter1985} and a Gr{\"o}nwall lemma allowing weakly singular kernels.

Our second main contribution is to establish pathwise uniqueness for the SVE~\eqref{eq:intro} provided that the diffusion coefficient~$\sigma$ is locally $1/2+\xi$-H{\"o}lder continuous for $\xi \in [0,1/2]$ or even, more generally, satisfies the classical Yamada--Watanabe condition~\cite{Yamada1971}. To that end, we generalize the classical approach of Yamada and Watanabe~\cite{Yamada1971} to the more general setting of stochastic Volterra equations. The presented proof for pathwise uniqueness is based on similar techniques as the proof of existence and is inspired by the work of Mytnik and Salisbury~\cite{Mytnik2015}. In~\cite{Mytnik2015}, pathwise uniqueness is proven for one-dimensional stochastic Volterra equations with smooth kernels and without drift (i.e. $\mu=0$). For SVEs of convolutional type with continuous differentiable kernels admitting a resolvent of the first kind, pathwise uniqueness was shown in \cite{AbiJaberElEuch2019b}.

Let us remark, while we need to require sufficient regularity on the kernels $K_\mu, K_\sigma$ to obtain the existence of a unique strong solution (see Theorem~\ref{thm:main result} and Corollary~\ref{cor:main result}), the imposed regularity conditions on the coefficients are essentially the classical regularity conditions of Yamada--Watanabe. Already in case of ordinary stochastic differential equations, it is well-known that these regularity conditions cannot be relaxed in the sense that pathwise uniqueness does not hold in general if, e.g., the diffusion coefficient~$\sigma$ is only H{\"o}lder continuous of order strictly less than~$1/2$.

\medskip

\noindent \textbf{Organization of the paper:} Section~\ref{sec:main result} presents the setting and main result: an existence and uniqueness theorem for stochastic Volterra equations with H{\"o}lder continuous diffusion coefficients. The properties of solutions to SVEs are provided in Section~\ref{sec:regularity}. The existence of a strong solution is proven in Section~\ref{sec:existence} and that pathwise uniqueness holds in Section~\ref{sec:pathwise uniqueness}.

\medskip

\noindent\textbf{Acknowledgments:} D. Scheffels gratefully acknowledges financial support by the Research Training Group ``Statistical Modeling of Complex Systems'' (RTG 1953) funded by the German Science Foundation (DFG).

\section{Main result and assumptions}\label{sec:main result}

Let $(\Omega,\mathcal{F},(\mathcal{F}_t)_{t\in [0,T]},\mathbb{P})$ be a filtered probability space, which satisfies the usual conditions, $(B_t)_{t\in [0,T]}$ be a standard Brownian motion and $T\in (0,\infty)$. We consider the one-dimensional stochastic Volterra equation (SVE)
\begin{equation}\label{eq:SVE}
  X_t = x_0(t)+\int_0^t K_{\mu}(s,t)\mu(s,X_s)\dd s+\int_0^t K_{\sigma}(s,t)\sigma(s,X_s)\dd B_s,\quad t\in [0,T],
\end{equation}
where $x_0\colon [0,T]\to \R$ is a continuous function, the coefficients $\mu,\sigma\colon [0,T]\times\R\to\R $ and the kernels $K_\mu, K_\sigma\colon \Delta_T\to \R$ are measurable functions, using the standard notation $\Delta_T~:=~\lbrace (s,t)\in [0,T]\times [0,T]\colon \, 0\leq s\leq t\leq T \rbrace$. Furthermore, $\int_0^t K_{\mu}(s,t)\mu(s,X_s)\dd s$ is defined as a Riemann--Stieltjes integral and $\int_0^t K_{\sigma}(s,t)\sigma(s,X_s)\dd B_s$ as an It{\^o} integral.

\medskip

Let $K\colon \Delta_T\to\R$ be a measurable function. We say $K(\cdot,t)$ is absolutely continuous for every $t\in [0,T]$ if there exists an integrable function $\partial_1 K\colon \Delta_T \to\R$ such that $K(s,t)-K(0,t)=\int_0^s \partial_1 K(u,t)\dd u$ for $(s,t)\in\Delta_T$. We say $K(s,\cdot)$ is absolutely continuous for every $s\in [0,T]$ if there exists an integrable function $\partial_2 K\colon \Delta_T \to\R$ such that $K(s,t)-K(s,0)=\int_0^t \partial_2 K(s,u)\dd u$ for $(s,t)\in\Delta_T$. Moreover, for $p\in [1,\infty)$, we denote $K\in L^p(\Delta_T)$ if $ \int_0^T\int_0^t |K(s,t)|^p \dd s\dd t<\infty$.

\medskip

For the kernels $K_\mu, K_\sigma$ and the initial condition $x_0$ we make the following assumptions.

\begin{assumption}\label{ass:kernels}
  Let $\gamma\in (0,\frac{1}{2}]$, and $K_\mu, K_\sigma\colon \Delta_T\to \R$ and $x_0\colon [0,T]\to \R$ be continuous functions such that:
  \begin{enumerate}
    \item[(i)] $K_\mu(s,\cdot)$ is absolutely continuous for every $s\in [0,T]$ and $\partial_2 K_\mu $ is bounded on~$\Delta_T$.

    \item[(ii)] $K_\sigma(\cdot,t)$ is absolutely continuous for every $t\in [0,T]$, $K_\sigma(s,\cdot)$ is absolutely continuous for every $s\in [0,T]$ with $\partial_2 K_\sigma\in L^{2}(\Delta_T)$, and $\partial_2 K_\sigma(\cdot,t)$ is  absolutely continuous for every $t\in[0,T]$. Furthermore, there is a constant $C>0$ such that $|K_\sigma(t,t)|\geq C$ for any $t\in [0,T]$, and there exist $C>0$, $\alpha\in [0,\frac{1}{2})$ and $\epsilon>0$ such that
    \begin{align*}
      &\int_0^s |K_{\sigma}(u,t)-K_{\sigma}(u,s)|^{2+\epsilon}\dd u \leq C|t-s|^{\gamma (2+\epsilon)}
      \text{ and}\\
      &|\partial_1 K_\sigma (s,t)| +|\partial_2 K_\sigma (s,s)| +\int_s^t |\partial_{21}K_\sigma(s,u)|\dd u\leq C(t-s)^{-\alpha}
    \end{align*}
    hold for any $(s,t)\in\Delta_T$.

    \item[(iii)] $x_0$ is $\beta$-H{\"o}lder continuous for every $\beta \in (0, \gamma)$.
  \end{enumerate}
\end{assumption}

The regularity properties of the coefficients~$\mu$ and~$\sigma$ are formulated in the next assumption. We start with assuming global Lipschitz and H{\"o}lder continuity of $\mu$ and~$\sigma$, respectively. An extension to local regularity conditions are treated in Corollary~\ref{cor:main result} below.

\begin{assumption}\label{ass:coefficients}
  Let $\mu,\sigma\colon [0,T]\times\R\to\R $ be measurable functions such that:
  \begin{enumerate}
    \item[(i)] $\mu$ and $\sigma$ are of linear growth, i.e. there is a constant $C_{\mu,\sigma}>0$ such that
    \begin{equation*}
      |\mu(t,x)|+|\sigma(t,x)|\leq C_{\mu,\sigma}(1+|x|),
    \end{equation*}
    for all $t\in [0,T]$ and $x\in\R$.
    \item[(ii)] $\mu$ is Lipschitz continuous and $\sigma$ is H{\"o}lder continuous of order $\frac{1}{2}+\xi$ for some $\xi\in [0,\frac{1}{2}]$ in the space variable uniformly in time, i.e. there are constants $C_\mu,C_\sigma>0$ such that
    \begin{equation*}
      |\mu(t,x)-\mu(t,y)|\leq C_\mu|x-y|
      \quad\text{and}\quad
      |\sigma(t,x)-\sigma(t,y)|\leq C_\sigma|x-y|^{\frac{1}{2}+\xi}
    \end{equation*}
    hold for all $t\in [0,T]$ and $x,y\in \R$.
  \end{enumerate}
\end{assumption}

To formulate our results, let us briefly recall the concepts of strong solutions and pathwise uniqueness. For this purpose, let $L^p(\Omega\times [0,T])$ be the space of all real-valued, $p$-integrable functions on $\Omega\times [0,T]$. We call an $(\mathcal{F}_t)_{t\in[0,T]}$-progressively measurable stochastic process $(X_t)_{t\in [0,T]}$ in $L^p(\Omega\times [0,T])$ on the given probability space $(\Omega,\mathcal{F},(\mathcal{F}_t)_{t\in[0,T]},\mathbb{P})$, a \textit{(strong) $L^p$-solution} of the SVE~\eqref{eq:SVE} if $ \int_0^t (|K_\mu(s,t)\mu(s,X_s)|+|K_\sigma(s,t)\sigma(s,X_s)|^2 )\dd s<\infty$ for all $t\in[0,T]$ and the integral equation~\eqref{eq:SVE} hold $\mathbb{P}$-almost surely. As usual, a strong $L^1$-solution $(X_t)_{t\in [0,T]}$ of the SVE~\eqref{eq:SVE} is often just called \textit{solution} of the SVE~\eqref{eq:SVE}. We say \textit{pathwise uniqueness} in $L^p(\Omega\times [0,T])$ holds for the SVE~\eqref{eq:SVE} if $\mathbb{P}(X_t=\tilde{X}_t, \,\forall t\in [0,T])=1$ for two $L^p$-solutions $(X_t)_{t\in[0,T]}$ and $(\tilde{X}_t)_{t\in[0,T]}$ of the SVE~\eqref{eq:SVE} defined on the same probability space $(\Omega,\mathcal{F},(\mathcal{F}_t)_{t\in[0,T]},\mathbb{P})$. Moreover, we say there exists a \textit{unique strong $L^p$-solution} $(X_t)_{t\in [0,T]}$ to the SVE~\eqref{eq:SVE} if $(X_t)_{t\in [0,T]}$ is a strong $L^p$-solution to the SVE~\eqref{eq:SVE} and pathwise uniqueness in $L^p(\Omega\times [0,T])$ holds for the SVE~\eqref{eq:SVE}. We say $(X_t)_{t\in [0,T]}$ is $\beta$-H{\"o}lder continuous for $\beta \in (0,1]$ if there exists a modification of $(X_t)_{t\in [0,T]}$ with sample paths that are $\mathbb{P}$-almost surely $\beta$-H{\"o}lder continuous.

\medskip

The main results of the present work are summarized in the following theorem. 

\begin{theorem}\label{thm:main result}
  Suppose Assumptions~\ref{ass:kernels} and~\ref{ass:coefficients}, and let $p>\max\{\frac{1}{\gamma},1+\frac{2}{\epsilon}\}$, where $\gamma\in(0,\frac{1}{2}]$ and $\epsilon>0$ are given by Assumption~\ref{ass:kernels}. Then, there exists a unique strong $L^p$-solution~$(X_t)_{t\in [0,T]}$ to the stochastic Volterra equation~\eqref{eq:SVE}. Moreover, the solution $(X_t)_{t\in [0,T]}$ is $\beta$-H{\"o}lder continuous for every $\beta \in (0, \gamma)$, $\sup_{t\in [0,T]}\E [|X_t|^q]<\infty$ for every $q\in [1,\infty)$ and $(X_t-x_0(t))_{t\in [0,T]}$ is a semimartingale.
\end{theorem}

\begin{proof}
  The existence of a strong solution~$(X_t)_{t\in [0,T]}$ to the stochastic Volterra equation~\eqref{eq:SVE} is provided by Theorem~\ref{thm:existence} and its pathwise uniqueness by Theorem~\ref{thm:uniqueness}. The assertions that $\sup_{t\in [0,T]}\E [|X_t|^q]<\infty$ for every $q\in [1,\infty)$ and of the $\beta$-H{\"o}lder continuity as well as the semimartingale property of $(X_t-x_0(t))_{t\in [0,T]}$ follow by Corollary~\ref{cor:regularity}.
\end{proof}

Note that the regularity assumptions (Assumption~\ref{ass:coefficients}), as required in Theorem~\ref{thm:main result}, on the coefficients~$\mu,\sigma$ are essentially optimal. Indeed, it is well-known for ordinary stochastic differential equations that pathwise uniqueness does not hold in general if $\mu$ is only H{\"o}lder continuous of order strictly less than $1$ or $\sigma$ is only H{\"o}lder continuous of order strictly less than~$1/2$, see for instance \cite[page~287]{Karatzas1991} and \cite[Chapter~5, Example~2.15]{Karatzas1991}.

\begin{remark}
  Recall that Yamada and Watanabe derived pathwise uniqueness for ordinary stochastic differential equations under the slightly weaker assumption of $|\sigma(t,x)-\sigma(t,y)|\leq \rho(|x-y|)$ for a function $\rho\colon [0,\infty) \to [0,\infty)$ with $\int_0^{\epsilon} \rho(s)^{-2}\dd s=\infty$ for every $\epsilon >0$, cf. \cite[Theorem~1]{Yamada1971}. While the proof of pathwise uniqueness presented in Section~\ref{sec:pathwise uniqueness} is given under this Yamada--Watanabe condition, in the proof of the existence of a strong solution via an approximation scheme the H{\"o}lder regularity of~$\sigma$ is explicitly used in various estimates, see e.g. \eqref{eq:YW_notenough}, and a modification of these estimates allowing for the Yamada--Watanabe condition appears not straightforward.
\end{remark}

\begin{remark}
  Assumption~\ref{ass:kernels} is satisfied, for instance, if $K_{\mu}$ is
  continuously differentiable, $K_{\sigma}$ is twice continuously differentiable with $K_{\sigma}(t,t)>0$ for $t\in [0,T]$ and $x_0$ is $\beta$-H{\"o}lder continuous for some $\beta \in (0,1)$.

  While the condition $|K_\sigma(t,t)|\geq C$ for $t\in [0,T]$ is crucial for implementing the present method to prove Theorem~\ref{thm:main result}, it might appear to be of technical nature. However, assuming $K_\sigma(t,t)=0$ for every $t\in [0,T]$ and keeping in mind the semimartingale decomposition in Lemma~\ref{lem:semimartingale property}, any solution of the SVE~\eqref{eq:SVE} would be a semimartingale of bounded variation without any diffusion part and, thus, some care is needed to not lose the regularization effects of a Brownian motion.
\end{remark}

Based on a localization argument, the assumptions of global Lipschitz and H\"older continuity on the coefficients of the SVE~\eqref{eq:SVE} can be relaxed to local regularity assumptions. In the following, $C>0$ denotes a generic constant that might change from line to line. To emphasize the dependence of the constant $C$ on parameters $p,q$ or functions $f,g$, we write $C_{p,q,f,g}$. Moreover, for $x,y\in \R$ we set $x\wedge y:= \min \{x,y\}$.

\begin{corollary}\label{cor:main result}
  Suppose Assumptions~\ref{ass:kernels}, \ref{ass:coefficients}~(i), and that $\mu$ is locally Lipschitz continuous and $\sigma$ is locally H{\"o}lder continuous of order $\frac{1}{2}+\xi$ for some $\xi\in [0,\frac{1}{2}]$ in the space variable uniformly in time, i.e. for every $n\in \N$ there are constants $C_{\mu,n},C_{\sigma,n}>0$ such that
  \begin{equation*}
    |\mu(t,x)-\mu(t,y)|\leq C_{\mu,n}|x-y|
    \quad\text{and}\quad
    |\sigma(t,x)-\sigma(t,y)|\leq C_{\sigma,n}|x-y|^{\frac{1}{2}+\xi}
  \end{equation*}
  hold for all $t\in [0,T]$ and $x,y\in \R$ with $|x|,|y|\leq n$. Let $p>\max\{\frac{1}{\gamma},1+\frac{2}{\epsilon}\}$, where $\gamma\in(0,\frac{1}{2}]$ and $\epsilon>0$ are given by Assumption~\ref{ass:kernels}. Then, there exists a unique strong $L^p$-solution~$(X_t)_{t\in [0,T]}$ to the stochastic Volterra equation~\eqref{eq:SVE}. Moreover, the solution $(X_t)_{t\in [0,T]}$ is $\beta$-H{\"o}lder continuous for every $\beta \in (0, \gamma)$, $\sup_{t\in [0,T]}\E [|X_t|^q]<\infty$ for every $q\in [1,\infty)$ and $(X_t-x_0(t))_{t\in [0,T]}$ is a semimartingale.
\end{corollary}

\begin{proof}
  By Assumptions~\ref{ass:kernels} and \ref{ass:coefficients}~(i), Lemma~\ref{lem:bound}, Corollary~\ref{cor_regularity} and Lemma~\ref{lem:semimartingale property} imply the integrability, $\beta$-H{\"o}lder continuity and semimartingale property of the solution. For the well-posedness, we adapt the proofs of Theorem~\ref{thm:existence} and ~\ref{thm:uniqueness} and the notation therein.

  For the uniqueness, consider two $L^p$-solutions $(X^1_t)_{t\in[0,T]}$ and $(X^2_t)_{t\in[0,T]}$, and define $\tilde{X}_t:=X^1_t-X_t^2$ for $t\in[0,T]$ and the hitting times $\tau_k:=\inf\lbrace t\in[0,T]\colon \max\lbrace |X_t|,|Y_t|\rbrace \geq k \rbrace\wedge T$ for $k\in\N$ which are stopping times with $\tau_k\to T$ a.s. by the same reasoning as for the hitting times defined in \eqref{def:hittingtimes}. By bounding $\phi_n(\tilde{X}_t \mathbbm{1}_{\lbrace t\leq \tau_k \rbrace})\leq \phi_n(\tilde{X}_{t\wedge\tau_k})$ and applying It{\^o}'s formula to the right-hand-side, we obtain after performing the same steps as in \eqref{eqmean}-\eqref{i4} and sending $n\to\infty$, that
  \begin{align*}
    &\E[|\tilde{X}_{t}|\mathbbm{1}_{\lbrace t\leq \tau_k \rbrace}]\\
    &\quad\leq  C \int_0^t\E[|\tilde{X}_{s}|\mathbbm{1}_{\lbrace s\leq \tau_k \rbrace}]\dd s
    + \int_0^t \E[|\tilde{Y}_{s}|\mathbbm{1}_{\lbrace s\leq \tau_k \rbrace}] \bigg( \partial_2 K_\sigma(s,s) + \int_s^t |\partial_{21}K_\sigma(s,u)|\dd u \bigg)\dd s,
  \end{align*}
  for $t\in[0,T]$. Similarly, we get a bound on $\E[|\tilde{Y}_{t}|\mathbbm{1}_{\lbrace t\leq \tau_k \rbrace}]$ analogue to \eqref{part2i} and denoting
  \begin{equation*}
    M_k(t):=\sup\limits_{s\in [0,t]}\left(\E[|\tilde{X}_{s}|\mathbbm{1}_{\lbrace s\leq \tau_k \rbrace}] + \E[|\tilde{Y}_{s}|\mathbbm{1}_{\lbrace s\leq \tau_k \rbrace}]\right)
  \end{equation*}
  we obtain $M_k(t)=0$ for all $t\in[0,T]$, and sending $k\to\infty$ yields the uniqueness.
  
  For the existence, we adapt the standard localization argument from the SDE case. We introduce for $n\in\N$ the localized coefficients
  \begin{equation*}
    \mu_n(t,x):=\begin{cases}
    \mu(t,x),\quad&\text{if }|x|\leq n,\\
    \mu(t,\frac{nx}{|x|}),\quad&\text{if }|x|>n,
    \end{cases}
  \end{equation*}
  and analogously $\sigma_n$, which fulfill the regularity properties globally, such that corresponding strong solutions exist by Theorem~\ref{thm:existence} that we denote by $X^n$. Moreover, let $\kappa_n:=\inf\lbrace t\in[0,T]\colon |X^n_t|>n \rbrace\wedge T$ and define $X(t):=X^n(t)$ for $\kappa_{n-1}<t\leq \kappa_{n}(t)$. By the pathwise uniqueness, it holds $X^{n-1}_{\tau_{n-1}}=X^n_{\tau_{n-1}}$ for all $n\in\N$ such that $X$ is continuously well-defined and we must only show that it cannot explode, i.e.~that $\kappa_n\to T$ a.s. By the Garsia--Rodemich--Rumsey inequality (see \cite[Lemma~1.1]{Garsia1970}), Markov's inequality and Lemma~\ref{lem:help_regularity}, we obtain for any $\alpha\in(0,\gamma)$ and $p>2$ chosen such that $\alpha p>1$ that
  \begin{align*}
    \mathbb{P}\big( \sup\limits_{t\in[0,T]}|X^n_t-X_0^n|>n \big) &\leq \mathbb{P}\bigg( \sup\limits_{t\in[0,T]}\Big( C_{\alpha,p}t^{\alpha-\frac{1}{p}}\Big( \int_0^T \int_0^T \frac{|X_s-X_u|^p}{|s-u|^{\alpha p+1}} \dd u\dd s\Big)^{\frac{1}{p}} \Big)> n \bigg)\\
    &\leq n^{-p} \mathbb{E}\bigg[ C_{\alpha,p,T}\Big( \int_0^T \int_0^T \frac{|X_s-X_u|^p}{|s-u|^{\alpha p+1}} \dd u\dd s\Big) \bigg]\\
    &\leq C_{\alpha,p,T,\mu,\sigma,\epsilon} n^{-p},
  \end{align*}
  which tends to $0$ sufficiently fast such that the Borel--Cantelli lemma (see \cite[Theorem~2.7]{Klenke2014}) implies $\kappa_n\to T$ a.s.
\end{proof}

The rest of the paper is largely devoted to prove Theorem~\ref{thm:main result}. However, we will formulate and prove the partial findings under weaker assumptions if possible without additional effort.

\section{Properties of a solution}\label{sec:regularity}

In this section we establish some properties of solutions to stochastic Volterra equations. We start by the regularity and integrability of $L^p$-solutions, which requires only the linear growth condition of the coefficients and allows for singular kernels in the SVE~\eqref{eq:SVE}.

\begin{lemma}\label{lem:help_regularity}
  Suppose Assumption~\ref{ass:coefficients}~(i) and let $K_{\mu},K_{\sigma}\colon \Delta_T\to \R$ be measurable functions such that, for some $\epsilon>0$ and $L>0$,
  \begin{align}\label{eq:regularity kernels}
  \begin{split}
    &\int_0^t |K_{\mu}(s,t')-K_{\mu}(s,t)|^{1+\epsilon}\dd s + \int_t^{t'} |K_{\mu}(s,t')|^{1+\epsilon}\dd s \leq L|t'-t|^{\gamma(1+\epsilon)},\\
    &\int_0^t |K_{\sigma}(s,t')-K_{\sigma}(s,t)|^{2+\epsilon}\dd s  + \int_t^{t'} |K_{\sigma}(s,t')|^{2+\epsilon}\dd s  \leq L|t'-t|^{\gamma(2+\epsilon)},
  \end{split}
  \end{align}
  for all $(t,t^\prime)\in \Delta_T$, and \eqref{eq:bound_qtilde} holds. Furthermore, let $x_0\colon[0,T]\to \R$ be $\beta$-H{\"o}lder continuous for every $\beta \in (0,\gamma)$ for some $\gamma\in (0,\frac{1}{2}]$ and let $(X_t)_{t\in [0,T]}$ be a $L^p$-solution of the SVE~\eqref{eq:SVE} for some $p>\max\lbrace \frac{1}{\gamma},1+\frac{2}{\epsilon} \rbrace$. Then, for any $\beta \in (0,\gamma)$, there is a constant $C_{x_0,p,L,T,\mu,\sigma,\epsilon}>0$ such that
  \begin{equation*}
    \mathbb{E}[|X_{t'}-X_t|^p]\leq C_{x_0,p,L,T,\mu,\sigma,\epsilon}|t'-t|^{\beta p},
  \end{equation*}
  holds for all $t, t'\in [0,T]$. Consequently, $(X_t)_{t\in [0,T]}$ is $\beta$-H{\"o}lder continuous for any $\beta\in (0, \gamma -\frac{1}{p})$.
\end{lemma}

\begin{proof}
  Let $p>2$ be given by the assumption. Since $x_0$ is $\beta$-H{\"o}lder continuous, we observe for $t,t^\prime\in [0,T]$ that
  \begin{equation*}
    \mathbb{E}[|X_{t'}-X_t|^p]
    \leq C_{p,x_0} |t'-t|^{\beta p} + C_p \mathbb{E}[|\tilde{X}_{t^{\prime}}-\tilde{X}_{t}|^p]
    \quad \text{with}\quad \tilde{X}_t:=X_{t}-x_0(t).
  \end{equation*}
  For $(t,t^\prime) \in \Delta_T$ we note that
  \begin{align*}
    |\tilde{X}_{t'}-\tilde{X}_t|^p
    =&\bigg|\int_0^{t'} K_\mu(s,t')\mu(s,X_s)\dd s+\int_0^{t'} K_\sigma(s,t')\sigma(s,X_s)\dd B_s\\
     &\quad-\int_0^t K_\mu(s,t)\mu(s,X_s)\dd s-\int_0^t K_\sigma(s,t)\sigma(s,X_s)\dd B_s\bigg|^p\\
     &\leq C_p\bigg(\bigg|\int_0^t \mu(s,X_s)( K_\mu(s,t')-K_\mu(s,t) )\dd s\bigg|^p + \bigg|\int_t^{t'} \mu(s,X_s) K_\mu(s,t')\dd s\bigg|^p\\
    &\quad+ \bigg|\int_0^t \sigma(s,X_s)( K_\sigma(s,t')-K_\sigma(s,t) )\dd B_s\bigg|^p + \bigg|\int_t^{t'} \sigma(s,X_s) K_\sigma(s,t')\dd B_s\bigg|^p\bigg)\\
    &=:C_p(A+B+C+D).
  \end{align*}
  We shall bound the expectation of the terms $A$-$D$ in the following. For $A$, we use H{\"o}lder's inequality, the linear growth of $\mu$ (Assumption~\ref{ass:coefficients}~(i)), \eqref{eq:regularity kernels} and that $X\in L^{\frac{1+\epsilon}{\epsilon}}(\Omega\times[0,T])$ since $\frac{1+\epsilon}{\epsilon}< p$ to obtain
  \begin{align*}
    \E[A]
    &\leq \E\bigg[\Big| \int_0^t|\mu(s,X_s)|^{\frac{1+\epsilon}{\epsilon}} \dd s\Big|^{\frac{p\epsilon}{1+\epsilon}}  \bigg] \bigg(\int_0^t\big|K_\mu(s,t')-K_\mu(s,t) \big|^{1+\epsilon}\dd s\bigg)^{\frac{p}{1+\epsilon}}\\
    &\leq C_{p,L,\mu,T,\epsilon} \bigg(\int_0^t\big|K_\mu(s,t')-K_\mu(s,t) \big|^{1+\epsilon}\dd s\bigg)^{\frac{p}{1+\epsilon}}\\
    &\leq C_{x_0,p,L,T,\mu,\sigma,\epsilon} |t'-t|^{\gamma p}.
  \end{align*}
  Note that the second inequality follows either with Jensen's inequality, if $\frac{p\epsilon}{1+\epsilon}\leq 1$, or else with H{\"o}lder's inequality and Fubini's theorem. Applying the analog estimates to $B$ gives
  \begin{equation*}
    \E[B] \leq \E\bigg[\Big|\int_t^{t^\prime}|\mu(s,X_s)|^{\frac{1+\epsilon}{\epsilon}}\dd s\Big|^{\frac{p\epsilon}{1+\epsilon}} \bigg] \bigg( \int_t^{t^\prime} \big|K_\mu(s,t') \big|^{1+\epsilon}\dd s\bigg)^{\frac{p}{1+\epsilon}}
    \leq C_{x_0,p,L,T,\mu,\sigma,\epsilon}|t'-t|^{\gamma p}.
  \end{equation*}
  For term $C$, relying on the Burkholder--Davis--Gundy inequality, H{\"o}lder's inequality, using the linear growth of $\sigma$ (Assumption~\ref{ass:coefficients}~(i)), $X\in L^{\frac{2+\epsilon}{\epsilon}}(\Omega\times[0,T])$ and \eqref{eq:regularity kernels}, we get
  \begin{align*}
    \E[C]
    &\leq \E \bigg[\bigg( \int_0^t \big| \sigma(s,X_s)\big( K_\sigma(s,t')-K_\sigma(s,t) \big) \big|^2\dd s\bigg)^{\frac{p}{2}} \bigg]\\
        &\leq \E\bigg[\Big| \int_0^t|\sigma(s,X_s)|^{\frac{2+\epsilon}{\epsilon}} \dd s\Big|^{\frac{p\epsilon}{4+2\epsilon}}  \bigg] \bigg(\int_0^t\big|K_\sigma(s,t')-K_\sigma(s,t) \big|^{2+\epsilon}\dd s\bigg)^{\frac{p}{2+\epsilon}}\\
        &\leq C_{p,L,\sigma,T,\epsilon} \bigg(\int_0^t\big|K_\sigma(s,t')-K_\sigma(s,t) \big|^{2+\epsilon}\dd s\bigg)^{\frac{p}{2+\epsilon}}\\
    &\leq  C_{x_0,p,L,T,\mu,\sigma,\epsilon} |t'-t|^{\gamma p}.
  \end{align*}
  Applying \eqref{eq:regularity kernels} and analog estimates to term~$D$ reveals
  \begin{equation*}
    \E[D]
    \leq C_{x_0,p,L,T,\mu,\sigma,\epsilon} \bigg(\int_t^{t'} K_\sigma(s,t')^{2+\epsilon}\dd s\bigg)^{\frac{p}{2+\epsilon}}
    \leq C_{x_0,p,L,T,\mu,\sigma,\epsilon} |t'-t|^{\gamma p}.
  \end{equation*}
  Hence, with the above estimates we arrive at
  \begin{equation*}
    \mathbb{E}[|X_{t'}-X_t|^p]
    \leq C_{p,x_0} |t'-t|^{\beta p} + C_{x_0,p,L,T,\mu,\sigma} |t'-t|^{\gamma p}
    \leq C_{x_0,p,L,T,\mu,\sigma,\epsilon} |t'-t|^{\beta p},
  \end{equation*}
  as $\beta < \gamma$. Hence, by Kolmogorov--Chentsov's theorem (see e.g. \cite[Theorem~21.6]{Klenke2014}) and sending $\beta\to\gamma$, there exists a modification of $(X_t)_{t\in [0,T]}$ which is $\delta^{\prime}$-H{\"o}lder continuous for $\delta^{\prime}\in (0,\gamma -1/p)$.
\end{proof}

\begin{remark}\label{rem_Kolmogorov}
  Suppose that the kernels $K_\mu$ and $K_\sigma$ fulfill Assumption~\ref{ass:kernels}. In this case it follows from Kolmogorov's continuity criterion and the estimates in the proof of Lemma~\ref{lem:help_regularity}, that, for every progressively measurable stochastic process $u\in L^p([0,T]\times\Omega)$ for some $p>\max\lbrace\frac{1}{\gamma},1+\frac{2}{\epsilon}\rbrace$, the process $(M_t^u)_{t\in[0,T]}$, defined by $M_t^u:=\int_0^t K_\mu(s,t)u_s\dd s+\int_0^t K_\sigma(s,t)u_s\dd B_s$, has $\mathbb{P}$-a.s. $\beta$-H{\"o}lder-continuous paths for every $\beta\in(0,\gamma-\frac{1}{p})$.
\end{remark}

\begin{remark}
  Note that the constant $C_{x_0,p,L,T,\mu,\sigma,\epsilon}$ in Lemma~\ref{lem:help_regularity} depends on $\mu$ and $\sigma$ only through the constant appearing in the linear growth condition (Assumption~\ref{ass:coefficients}~(i)).
\end{remark}

The integrability of solutions to the SVE~\eqref{eq:SVE} is the content of the next lemma.

\begin{lemma}\label{lem:bound}
  Suppose Assumption~\ref{ass:coefficients}~(i) and that $K_{\mu},K_{\sigma}\colon \Delta_T\to \R$ are measurable functions such that, for some $\epsilon>0$ and $L>0$,
  \begin{equation}\label{eq:bound_qtilde}
    \int_0^t |K_\mu(s,t)|^{1+\epsilon}\dd s+\int_0^t |K_\sigma(s,t)|^{2+\epsilon}\dd s\leq L,\quad t\in [0,T].
  \end{equation}
  Let $(X_t)_{t\in [0,T]}$ be a $L^p$-solution to the SVE~\eqref{eq:SVE} for some $p>\max\lbrace 2,1+\frac{2}{\epsilon} \rbrace$. Then,
  \begin{equation*}
    \sup\limits_{t\in[0,T]} \E[|X_t|^q]\leq C_{q,L,T,\mu,\sigma} \bigg(1+\sup\limits_{t\in [0,T]}|x_0(t)|^q \bigg),
  \end{equation*}
 holds for any $q\geq 1$, where the constant $ C_{q,L,T,\mu,\sigma}$ depends only on $q$, $L$, $T$ and the growth constants of $\mu$ and $\sigma$.
\end{lemma}

\begin{proof}
  Let us introduce the hitting times
  \begin{equation}\label{def:hittingtimes}
    \tau_k := \inf\lbrace t\in [0,T]\colon |X_t|\geq k \rbrace\wedge T,\quad \text{for }k\in \N.
  \end{equation}
  Note that $\tau_k\to T$ a.s. as $k\to\infty$, since the paths of the solution $X$ are $\mathbb{P}$-a.s. continuous by Lemma~\ref{lem:help_regularity}. Since the underlying filtered probability space satisfies the usual conditions, by the D{\'e}but theorem (see \cite[Chapter~I, (4.15)~Theorem]{Revuz1999}), the hitting times $(\tau_k)_{k\in\N}$ are stopping times.

  First, let $q>2$ be big enough such that $q':=\frac{q}{q-1}\leq 1+\epsilon$ and $\tilde{q}:=\frac{q}{q-2}\leq 1+\epsilon/2$. Using H{\"o}lder's inequality, the Burkholder--Davis--Gundy inequality, and the linear growth condition (Assumption~\ref{ass:coefficients}~(i)), we get
  \begin{align}\label{ineq:derivation_momentbound}
    &\E [|X_{t}|^q \1_{\lbrace t\leq\tau_k \rbrace}]\notag\\
    &\quad=\E\bigg[ \bigg| x_0({t})+\int_0^{t} K_\mu(s,{t})\mu(s,X_s)\dd s +\int_0^{t} K_\sigma(s,{t})\sigma(s,X_s)\dd B_s \bigg|^q \1_{\lbrace t\leq\tau_k \rbrace} \bigg]\notag\\
    &\quad = \E\bigg[ \bigg| x_0({t})\,\1_{\lbrace t\leq\tau_k \rbrace}+\int_0^{t} K_\mu(s,{t})\mu(s,X_s)\dd s\,\1_{\lbrace t\leq\tau_k \rbrace} +\int_0^{t} K_\sigma(s,{t})\sigma(s,X_s) \dd B_s \,\1_{\lbrace t\leq\tau_k \rbrace}\bigg|^q  \bigg]\notag\\    
    &\quad\leq C_q\E\bigg[ \big| x_0({t})\big|^q+\Big|\int_0^{t} K_\mu(s,{t})\mu(s,X_s)\,\1_{\lbrace s\leq\tau_k \rbrace}\dd s\Big|^q +\Big|\int_0^{t} K_\sigma(s,{t})\sigma(s,X_s)\,\1_{\lbrace s\leq\tau_k \rbrace} \dd B_s \big|^q  \bigg]\notag\\
    &\quad\leq C_q\bigg(|x_0({t})|^q+\bigg( \int_0^{t} |K_\mu(s,{t})|^{q'}\dd s \bigg)^{\frac{q}{q'}}\int_0^{t} \E[|\mu(s,X_s)|^q \1_{\lbrace s\leq\tau_k \rbrace}]\dd s\notag\\
    &\quad\qquad\qquad +\E\bigg[\bigg(\int_0^{t} | K_\sigma(s,{t})\sigma(s,X_s) |^2 \1_{\lbrace s\leq\tau_k \rbrace} \dd s\bigg)^{\frac{q}{2}}\bigg]\bigg) \notag\\
    &\quad\leq C_q\bigg(|x_0({t})|^q+C_{q,\mu}\bigg( \int_0^{t} |K_\mu(s,{t})|^{q'}\dd s \bigg)^{\frac{q}{q'}}\int_0^{t} \E[1+|X_s|^q \1_{\lbrace s\leq\tau_k \rbrace}]\dd s \notag\\
    &\quad\qquad\qquad +C_{q,\sigma}\bigg( \int_0^{t} |K_\sigma(s,{t})|^{2\tilde{q}}\dd s \bigg)^{\frac{q}{2\tilde{q}}}\int_0^{t} \E[1+| X_s|^q \1_{\lbrace s\leq\tau_k \rbrace}] \dd s\bigg)
  \end{align}  
  for $t\in [0,T]$. Due to \eqref{eq:bound_qtilde} we arrive at
  \begin{equation*}
    \E [|X_{t}|^q \1_{\lbrace t\leq\tau_k \rbrace}]
    \leq C_{q,L,T,\mu,\sigma} \bigg(1+|x_0({t})|^q+\int_0^{t} \E[|X_{s}|^q \1_{\lbrace s\leq\tau_k \rbrace}]\dd s \bigg)
  \end{equation*}
  and, thus, as $t\mapsto \E[|X_{t}|^q \1_{\lbrace t\leq\tau_k \rbrace}]$ is bounded by $k$ on $[0,T]$, we can apply Gr{\"o}nwall's lemma (see e.g. \cite[Lemma~26.9]{Klenke2014}) to get
  \begin{equation*}
    \E [|X_{t}|^q \1_{\lbrace t\leq\tau_k \rbrace}]\leq C_{q,L,T,\mu,\sigma} \bigg(1+\sup\limits_{t\in [0,T]}|x_0(t)|^q \bigg), \quad t\in [0,T].
  \end{equation*}
  Sending $k\to\infty$ and taking the supremum over $[0,T]$ reveals the assertion. The orderedness of the $L^p$-spaces implies the statement also for $q_2\in [1,q)$.
\end{proof}

We conclude that the regularity of a solution can be improved.

\begin{corollary}\label{cor_regularity}
  Under the assumptions of Lemma~\ref{lem:help_regularity}, any $L^p$-solution to the SVE~\eqref{eq:SVE} for some $p>\max\lbrace \frac{1}{\gamma},1+\frac{2}{\epsilon} \rbrace$ is $\beta$-H{\"o}lder continuous for any $\beta\in(0,\gamma)$.
\end{corollary}

\begin{proof}
  The statement follows by applying Lemma~\ref{lem:bound} and Lemma~\ref{lem:help_regularity} with $q>2$ and then sending $q\to\infty$.
\end{proof}

Assuming sufficient regularity of the kernels $K_\mu,K_{\sigma}$, every solution to the stochastic Volterra equation~\eqref{eq:SVE} is essentially a semimartingale as first observed in \cite[Theorem~3.3]{Protter1985}.

\begin{lemma}\label{lem:semimartingale property}
  Let $K_\mu,K_{\sigma}\colon\Delta_T \to \R$ be measurable functions. Suppose $K_\mu(s,\cdot)$ is absolutely continuous for every $s\in [0,T]$ with $\partial_2 K_\mu \in L^1(\Delta_T)$, $K_{\sigma}(s,\cdot)$ is absolutely continuous for every $s\in [0,T]$ with $\partial_2 K_\sigma\in L^{2}(\Delta_t)$, and Assumption~\ref{ass:coefficients}~(i) holds. Let $(X_t)_{t\in [0,T]}$ be a solution to the SVE~\eqref{eq:SVE} such that $\E[|X_t|^2]\leq C$ for all $t\in [0,T]$ and some constant~$C$. Then, $(X_t-x_0(t))_{t\in [0,T]}$ is a semimartingale with decomposition $X_t-x_0(t)=M_t+A_t$ where
  \begin{align*}
    &M_t:= \int_0^t K_\sigma(s,s)\sigma(s,X_s)\dd B_s\quad \text{and}\\
    &A_t:= \int_0^t K_\mu(s,s)\mu(s,X_s) \dd s \\
    &\qquad +  \int_0^t \bigg(\int_0^s \partial_2 K_\mu(u,s)\mu(u,X_u) \dd u + \int_0^s \partial_2 K_\sigma(u,s)\sigma(u,X_u)\dd B_u\bigg) \dd s
  \end{align*}
  for $t\in [0,T]$.
\end{lemma}

\begin{proof}
  Setting
  \begin{equation*}
    Y_t:=\int_0^t \sigma(s,X_s)\dd B_s
     \quad\text{and}\quad
    Z_t:=\int_0^t \mu(s,X_s) \dd s,\quad\text{for } t\in [0,T],
  \end{equation*}
  and using the absolute continuity of $K_{\mu},K_{\sigma}$, we get
  \begin{align*}
    X_t
    &=\int_0^t K_\mu(s,s)\dd Z_s+ \int_0^t \Big(\int_s^t \partial_2 K_\mu(s,u)\dd u\Big)\dd Z_s\\
    &\qquad + \int_0^t \left( \int_s^t \partial_2 K_\sigma(s,u)\dd u\right) \dd Y_s + \int_0^t K_\sigma(s,s)\dd Y_s.
  \end{align*}
  Since
  \begin{equation*}
    \E \bigg[ \int_{\Delta_T}|\partial_2 K_\mu(s,u)\mu(s,X_s)|\dd s\dd u  \bigg]
    + \E \bigg[ \int_{\Delta_T} (\partial_2 K_\sigma(s,u)\sigma(s,X_s))^2\dd s\dd u  \bigg]
    <\infty
  \end{equation*}
  due to $\E[|X_t|^2]\leq C$ for all $t\in [0,T]$, $\partial_2 K_\mu\in L^1(\Delta_T)$ and $\partial_2 K_\sigma \in L^2(\Delta_T)$, we can apply the classical and the stochastic Fubini theorem (see e.g. \cite[Theorem~2.2]{Veraar2012}) to get
  \begin{align*}
    X_t
    &=\int_0^t K_\mu(s,s)\dd Z_s+ \int_0^t \Big(\int_0^u \partial_2 K_\mu(s,u)\dd Z_s\Big)\dd u\\
    &\qquad+ \int_0^t \left( \int_0^u \partial_2 K_\sigma(s,u)\dd Y_s\right) \dd u + \int_0^t K_\sigma(s,s)\dd Y_s,
  \end{align*}
  which completes the proof.
\end{proof}

Applying the previous lemmas to the setting of Theorem~\ref{thm:main result} leads to the following corollary.

\begin{corollary}\label{cor:regularity}
  Suppose Assumptions~\ref{ass:kernels} and~\ref{ass:coefficients}. Let $(X_t)_{t\in [0,T]}$ be a $L^p$-solution to the SVE~\eqref{eq:SVE} for some $p>\max\lbrace \frac{1}{\gamma},1+\frac{2}{\epsilon} \rbrace$. Then, $(X_t)_{t\in [0,T]}$ satisfies $\sup_{t\in [0,T]}\E [|X_t|^q]<\infty$ for every $q\in [1,\infty)$, $(X_t)_{t\in [0,T]}$ is $\beta$-H{\"o}lder continuous for every $\beta \in (0,\gamma)$ for $\gamma \in (0,1/2]$ given in Assumption~\ref{ass:kernels}, and $(X_t-x_0(t))_{t\in [0,T]}$ is a semimartingale.
\end{corollary}

\begin{proof}
  Note that the existence and boundedness of $\partial_2K_\mu$ from Assumption~\ref{ass:kernels}~(i) imply that
  \begin{align*}
    \int_0^s |K_{\mu}(u,t)-K_{\mu}(u,s)|^{1+\epsilon}\dd u &= \int_0^s \big| \int_s^t \partial_2 K_{\mu}(u,r)\dd r\big|^{1+\epsilon}\dd u\\
    & \leq C|t-s|^{\gamma}\notag
  \end{align*}
  holds for some $C>0$ and any $(s,t)\in \Delta_T$, using $\epsilon>0$ and $\gamma\in (0,1/2]$ from Assumption~\ref{ass:kernels}. Furthermore, the continuity of $K_{\mu}$ and $K_{\sigma}$ ensures that condition~\eqref{eq:bound_qtilde} holds and, thus, $\sup_{t\in [0,T]}\E [|X_t|^q]<\infty$ for every $q\in [1,\infty)$ by Lemma~\ref{lem:bound}. Moreover, since Assumption~\ref{ass:kernels} implies \eqref{eq:regularity kernels}, Corollary~\ref{cor_regularity} states the claimed $\beta$-H{\"o}lder continuity. The semimartingale property follows by Lemma~\ref{lem:semimartingale property}.
\end{proof}

\section{Existence of a strong solution}\label{sec:existence}

This section is devoted to establish the existence of a strong solution to the SVE~\eqref{eq:SVE}:

\begin{theorem}\label{thm:existence}
  Suppose Assumptions~\ref{ass:kernels} and~\ref{ass:coefficients}, and let $p>\max\lbrace \frac{1}{\gamma},1+\frac{2}{\epsilon} \rbrace$. Then, there exists a strong $L^p$-solution $(X_t)_{t\in [0,T]}$ to the SVE~\eqref{eq:SVE}.
\end{theorem}

The construction of a strong solution relies on an Euler type approximation. To set up the approximation, we use the sequence $(\pi_{m})_{m\in\N}$ of partitions defined by
\begin{equation*}
  \pi_m:=\lbrace t_0^m,\dots,t_{2^{m^5}}^{m}\}
  \quad \text{with}\quad
  t^m_i:= \frac{iT}{2^{m^5}}\quad \text{for } i=0,\dots,2^{m^5}
\end{equation*}
and introduce, for every $m\in \N$, the function $\kappa_m\colon [0,T]\to [0,T]$ by
\begin{equation*}
  \kappa_m(T):=T \quad \text{and}\quad
  \kappa_m(t):=t_{i}^m\quad\text{for } t_i^m \leq t < t^m_{i+1},\quad\text{for }i=0,1,\dots,2^{m^5}-1.
\end{equation*}
For every $m\in\N$, we iteratively define the process $(X^m(t))_{t\in [0,T]}$ by $X^m(0):=x_0(0)$ and for $t\in (t^m_i,t^m_{i+1}]$ by
\begin{align*}
  X^m(t):=&x_0(t)+\int_0^{t^m_i} K_\mu(s,t)\mu(s,X^m(\kappa_m(s)))\dd s
  +\int_{t^m_i}^{t} K_\mu(s,t)\mu(s,X^m(t^m_{i}))\dd s \\
  &+\int_0^{t^m_i} K_\sigma(s,t)\sigma(s,X^m(\kappa_m(s)))\dd B_s
  +\int_{t^m_i}^{t} K_\sigma(s,t)\sigma(s,X^m(t^m_{i}))\dd B_s,
\end{align*}
for $i=0,\dots,2^{m^5}-1$.

Note that we neither discretize the kernels~$K_{\mu},K_{\sigma}$ nor the time-component in the coefficients~$\mu,\sigma$. While these additional discretizations might be desirable to derive an implementable numerical scheme, for our purpose of proving the existence of a strong solution, it is sufficient to avoid this additional approximation.

\begin{lemma}\label{lem:XmL1}
  Suppose Assumptions~\ref{ass:kernels} and~\ref{ass:coefficients}. $X^m\in L^q(\Omega\times [0,T])$ for every $m\in\N$ and any $q\in [1,\infty)$. In particular, $X^m\in L^{p}(\Omega\times [0,T])$ for every $m\in\N$ and $p>\max\lbrace \frac{1}{\gamma},1+\frac{2}{\epsilon} \rbrace$.
\end{lemma}

\begin{proof}
  For $m\in \N$ and $q\in (2,\infty)$ we define
  \begin{equation*}
    g_m(t):=\E\big[|X^m(t)|^q\big]
    \quad \text{for } t\in [0,T].
  \end{equation*}
  To prove that $X^m\in L^q(\Omega\times [0,T])$, it is sufficient to show that the function $g_m$ is bounded on $[0,T]$ since
  \begin{equation*}
    \E\bigg[\int_0^T |X^m(t)|^q \dd t \bigg]
    = \int_0^T g_m(t) \dd t\leq T \sup_{t\in [0,T]} g_m(t).
  \end{equation*}

  For $t=0$ we have $\E[|X^m(0)|^q]=|x_0(0)|^q<\infty$ and, thus, $g_m$ is bounded on $[0,t_1^m]$. For $t\in (t^m_i,t^m_{i+1}]$ with $i= 1,\dots, 2^{m^5}-1$, using similar estimates as in \eqref{ineq:derivation_momentbound}, we iteratively get that
  \begin{align*}
    &\E[|X^m(t)|^q]\notag\\
    & \quad\leq C \bigg(|x_0(t)|^q\notag\\
    &\qquad+ \E\bigg[ \bigg| \int_0^{t_{i}^m}K_\mu(s,t)\mu(s,X^m(\kappa_m(s)))\dd s \bigg|^q \bigg]+ \E\bigg[ \bigg| \int_{t_{i}^m}^{t}K_\mu(s,t)\mu(s,X^m(t_{i}^m))\dd s \bigg|^q \bigg]\notag\\
    &\qquad + \E\bigg[ \bigg| \int_0^{t_{i}^m}K_\sigma(s,t)\sigma(s,X^m(\kappa(s)))\dd B_s \big|^q \bigg] + \E\bigg[ \bigg| \int_{t^m_i}^{t}K_\sigma(s,t)\sigma(s,X^m(t_{i}^m))\dd B_s \big|^q \bigg] \bigg)\notag\\
    & \quad\leq C \bigg(|x_0(t)|^q
    + \int_0^{t_{i}^m} \E \big [|\mu(s,X^m(\kappa_m(s)))|^q\big]\dd s
    + \int_{t_{i}^m}^{t}\E\big [|\mu(s,X^m(t_{i}^m))|^q\big]\dd s \notag\\
    &\qquad +  \int_0^{t_{i}^m} \E\big[ |\sigma(s,X^m(\kappa(s)))|^q\big]\dd s +
    \int_{t^m_i}^{t} \E\big[ |\sigma(s,X^m(t_{i}^m))|^q\big]\dd s \bigg)\notag\\
    &\quad \leq C\bigg(1 + \int_0^{t_{i}^m}\E[|X^m(\kappa(s))|^q]\dd s+ \int_{t_{i}^m}^t\E[|X^m(t_{i}^m)|^q]\dd s\bigg) <\infty.
  \end{align*}
  Therefore, $\sup_{t\in [0,T]} g_m(t)<\infty$.
\end{proof}

It can be quickly seen that the integrability and regularity results from Section~\ref{sec:regularity} also hold for the process~$(X^m(t))_{t\in[0,T]}$.

\begin{proposition}\label{prop:Xm}
  Suppose Assumptions \ref{ass:kernels} and \ref{ass:coefficients}. Let $\gamma\in [0,1/2]$ be as given in Assumption~\ref{ass:kernels}. Then, for any $m\in\N$, there is a constant $C>0$ such that
  \begin{equation*}
    \sup\limits_{t\in [0,T]}\E[|X^m(t)|^q]\leq C \bigg(1+\sup\limits_{t\in [0,T]}|x_0(t)|^q\bigg).
  \end{equation*}
  holds for any $q\geq 1$.
  Moreover, for any $\beta\in (0,\gamma)$, there is a constant $C>0$ such that
  \begin{equation*}
    \mathbb{E}[|X^m(t')-X^m(t)|^q]\leq C |t'-t|^{\beta q}
  \end{equation*}
  holds for all $t',t\in [0,T]$. Consequently, $(X^m (t))_{t\in[0,T]}$ is $\beta$-H{\"o}lder continuous for any $\beta\in (0,\gamma)$.
\end{proposition}

\begin{proof}
  The $L^q$-bound of $(X^m(t))_{t\in [0,T]}$ follows by similar arguments as used in the proof of Lemma~\ref{lem:bound}.

  For $t\in (t_i^m,t_{i+1}^m]$ and fixed $m\in \N$ and $q\geq 2$, we get
  \begin{align*}
    &\E [|X^m(t)|^q]\leq C \bigg(|x_0(t)|^q+\int_0^{t_i^m} \E[|X^m(\kappa_m(s))|^q] \dd s+\int_{t_i^m}^t \E[|X^m(t_i^m)|^q] \dd s \bigg),
  \end{align*}
  where we used H{\"o}lder's inequality, Burkholder--Davis--Gundy's inequality, and the linear growth condition (Assumption~\ref{ass:coefficients}~(i)). Hence, we arrive at
  \begin{equation*}
    \sup_{u\in [0, t]} \E [|X^m(u)|^q]
    \leq C \bigg( \sup_{u\in [0,T]} |x_0(u)|^q+\int_0^{t} \sup_{u\in [0, s]} \E[|X^m(u)|^q]\dd s\bigg).
  \end{equation*}
  Since $t \mapsto \sup_{u\in [0, t]} \E [|X^m(u)|^q]$ is bounded by the proof of Lemma~\ref{lem:XmL1}, we can apply Gr{\"o}nwall's lemma (see e.g. \cite[Lemma~26.9]{Klenke2014}) to get
  \begin{equation*}
     \sup_{t\in [0, T]} \E [|X^m(t)|^q ] \leq C \bigg(1+\sup\limits_{ t \in [0,T]}|x_0(t)|^q \bigg), \quad t\in [0,T],
  \end{equation*}
  which reveals the assertion.

  The regularity statement follows by adapting the proof of Lemma~\ref{lem:help_regularity}. Indeed, the regularity assumption on the kernels (Assumption~\ref{ass:kernels}) yields that condition~\eqref{eq:regularity kernels} is fulfilled. Thus, performing similar estimations as in the proof of Lemma~\ref{lem:help_regularity} and using the just established $L^q$-bound of $X^m$, we obtain
  \begin{equation*}
    \mathbb{E}[|X^m(t')-X^m(t)|^q]\leq C |t'-t|^{\beta q},
  \end{equation*}
  for $\beta\in (0,\gamma)$. Hence, by Kolmogorov--Chentsov's theorem (see e.g. \cite[Theorem~21.6]{Klenke2014}), there exists a modification of $(X^m(t))_{t\in [0,T]}$ which is $\delta^{\prime}$-H{\"o}lder continuous for $\delta^{\prime}\in (0,\beta -1/q)$. Sending $\beta \to \gamma $ and $q\to\infty$ leads to the claimed H{\"o}lder regularity.
\end{proof}

Due to Proposition~\ref{prop:Xm}, for every $m\in\N$ the process $(X^m(t))_{t\in[0,T]}$ has a continuous modification. Hence, keeping the definition of $(X^m(t))_{t\in[0,T]}$ in mind, we see that $(X^m(t))_{t\in[0,T]}$ fulfills the integral equation
\begin{equation}\label{eq:X_n}
  X^m(t)=x_0(t)+\int_0^t K_\mu(s,t)\mu(s,X^m(\kappa_m(s)))\dd s +\int_0^t K_\sigma(s,t)\sigma(s,X^m(\kappa_m(s)))\dd B_s,
\end{equation}
for $t\in [0,T]$. Moreover, using the just derived regularity estimates of $(X^m(t))_{t\in[0,T]}$, we obtain the following bound.

\begin{corollary}\label{cor:bound}
  Suppose Assumptions \ref{ass:kernels} and \ref{ass:coefficients}. Then, for any $q,\delta\in (0,\infty)$, there is a constant $C>0$ such that
  \begin{equation*}
    \E\left[\left( \int_0^T|X^m(s)-X^m(\kappa_m(s))|^\delta \dd s \right)^q\right]\leq C2^{-\delta q\beta m^{5}},
  \end{equation*}
  holds for all $\beta\in (0,\gamma)$ and $m\in\N$.
\end{corollary}

\begin{proof}
  Let $\delta>0$ be fixed. First, assume $q\geq1$ is sufficiently large such that $q\delta>2$. For $\beta\in (0,\gamma)$ and $m\in\N$, we use H{\"o}lder's inequality, Fubini's theorem and Proposition~\ref{prop:Xm} to get
  \begin{align}\label{eq:integral estmate approximation}
    \E\bigg[\bigg( \int_0^T|X^m(s)-X^m(\kappa_m(s))|^\delta \dd s \bigg)^q\bigg]
    &\leq C \E\bigg[ \int_0^T|X^m(s)-X^m(\kappa_m(s))| ^{\delta q} \dd s\bigg] \nonumber\\
    &= C \int_0^T\E\bigg[|X^m(s)-X^m(\kappa_m(s))|^{\delta q}\bigg] \dd s\nonumber\\
    &\leq C\int_0^T|s-\kappa_m(s)|^{\delta q \beta}\dd s\nonumber\\
    &\leq C 2^{-\delta q\beta m^{5}}.
  \end{align}
  For $0 < q\leq \frac{2}{\delta}$, we choose $q^\prime >q$ is sufficiently large such that $q^\prime \delta>2$. Applying Jensen's inequality and \eqref{eq:integral estmate approximation}, we obtain
  \begin{align*}
    \E\bigg[\bigg( \int_0^T|X^m(s)-X^m(\kappa_m(s))|^\delta \dd s \bigg)^q\bigg]
    &\leq C \E\bigg[\bigg(\int_0^T|X^m(s)-X^m(\kappa_m(s))|^{\delta } \dd s \bigg)^{q^\prime} \bigg]^{\frac{q}{q^\prime}}\nonumber\\
    &\leq C 2^{-\delta q\beta m^{5}}.
  \end{align*}
\end{proof}

\begin{lemma}\label{lem:1}
  Suppose Assumptions~\ref{ass:kernels} and~\ref{ass:coefficients}. Then, there is a sequence $(C_m)_{m\in\N}$ of constants such that
  \begin{equation*}
    \E[|X^{m+1}(t)-X^m(t)|]
    \leq C_m
  \end{equation*}
  holds for every $t\in[0,T]$, and $\sum_{m=1}^\infty C_m^{1/4}<
  \infty$.
\end{lemma}

\begin{proof}
  Following~Gy{\"o}ngy--R{\'a}sonyi~\cite{Gyongy2011} and Yamada--Watanabe~\cite{Yamada1971}, we approximate the function $\phi(x):=|x|$ by smooth functions $\phi_{\delta\epsilon}(x)$ for $\delta>1$ and $\epsilon>0$. To that end, note that
  \begin{equation*}
    \int_{\frac{\epsilon}{\delta}}^{\epsilon}\frac{1}{x}\dd x=\ln(\delta),
  \end{equation*}
  and, thus, there is a continuous, non-negative function $\psi_{\delta \epsilon}\colon \R_+\to \R_+$, that is zero outside the interval $[\frac{\epsilon}{\delta},\epsilon]$, $\int_0^{\infty}\psi_{\delta\epsilon}(x)\dd x=1$ and satisfies
  \begin{equation*}
    \psi_{\delta \epsilon}(x)\leq \frac{2}{x \ln(\delta)}.
  \end{equation*}
  We define
  \begin{equation*}
    \phi_{\delta \epsilon}(x):=\int_0^{|x|}\int_0^y \psi_{\delta \epsilon}(z)\dd z\dd y
    \quad \text{for } x\in\R,
  \end{equation*}
  such that the inequalities 
  \begin{equation}\label{ineq:phi}
    |x|\leq \phi_{\delta \epsilon}(x)+\epsilon,\quad
    0\leq |\phi_{\delta\epsilon}'(x)|\leq 1
    \quad\text{and}\quad
    \phi_{\delta\epsilon}''(x)=\psi_{\delta\epsilon}(|x|)\leq \frac{2}{|x|\ln(\delta)}\1_{[\frac{\epsilon}{\delta},\epsilon]}(|x|)
  \end{equation}
  hold for all $x\in\R$, where $\1_{[\frac{\epsilon}{\delta},\epsilon]}$ denotes the indicator function of the interval $[\frac{\epsilon}{\delta},\epsilon]$.

  To apply It{\^o}'s formula to $\phi_{\delta\epsilon}(\tilde{X}^m_t)$, where
  \begin{equation*}
    \tilde{X}^m_t:=X^{m+1}(t)-X^m(t),\quad t\in [0,T],
  \end{equation*}
  we need to find the semimartingale decomposition of $(\tilde{X}^m_t)_{t\in[0,T]}$. For this purpose, we introduce the local martingale
  \begin{equation*}
    \tilde{Y}^m_t:=Y^{m+1}_t-Y^m_t
    \quad \text{with}\quad Y^m_t:=\int_0^t\sigma(s,X^m(\kappa_m(s)))\dd B_s
  \end{equation*}
  and the process of finite variation
  \begin{equation*}
    \tilde{Z}_t^m:=\int_0^t \mu(s,X^{m+1}(\kappa_{m+1}(s)))\dd s - \int_0^t \mu(s,X^{m}(\kappa_{m}(s)))\dd s,\quad \text{for } t\in [0,T].
  \end{equation*}
  Since $\partial_2 K_{\mu}\in L^{1}(\Delta_T)$, $\partial_2 K_\sigma\in L^{2}(\Delta_T)$ (see Assumption~\ref{ass:kernels}) and the integrability property of $(X^m(t))_{\in [0,T]}$ as presented in Proposition~\ref{prop:Xm}, we obtain, as in the proof of Lemma~\ref{lem:semimartingale property}, the following semimartingale decomposition
  \begin{align*}
    \tilde{X}^m_t&= \int_0^t K_\mu(s,t)\dd \tilde{Z}_s^m+\int_0^t K_\sigma(s,t)\dd \tilde{Y}_s^m\notag\\
    &=\int_0^t K_\mu(s,s)\dd \tilde{Z}_s^m+\int_0^t \Big(\int_0^s \partial_2 K_\mu(u,s)\dd \tilde{Z}_u^m\Big)\dd s\notag\\
    &\qquad + \int_0^t \tilde{H}_s^m \dd s + \int_0^tK_\sigma(s,s)\dd \tilde{Y}_s^m,
  \end{align*}
  where $\tilde{H}^m_t:=H^{m+1}_t-H^m_t$ with $H_t^m:=\int_0^t\partial_2 K_\sigma(s,t)\dd Y_s^m$. Note that the quadratic variation of $(\tilde{X}^m_t)_{t\in [0,T]}$ is given by
  \begin{align*}
    \langle\tilde{X}^m\rangle_t
    &=\bigg\langle \int_0^{\cdot}K_\sigma(s,s)\Big(\sigma(s,X^{m+1}\big(\kappa_{m+1}(s))\big)-\sigma\big(s,X^m(\kappa_m(s))\big)\Big)\dd B_s \bigg\rangle_t\\
    &=\int_0^t K_\sigma(s,s)^2\Big(\sigma\big(s,X^{m+1}(\kappa_{m+1}(s))\big)-\sigma\big(s,X^m(\kappa_m(s))\big)\Big)^2\dd s,\quad t\in[0,T].
  \end{align*}
  Hence, using \eqref{ineq:phi} and applying It{\^o}'s formula for fixed $\epsilon>0$ and $\delta>1$ yields
  \begin{align}\label{eqmean.1}
    |\tilde{X}^m_t|&\leq \epsilon+\phi_{\delta\epsilon}(\tilde{X}^m_t)\notag\\
    &=\epsilon+\int_0^t\phi_{\delta\epsilon}'(\tilde{X}^m_s)\dd \tilde{X}^m_s+\frac{1}{2}\int_0^t \phi_{\delta\epsilon}''(\tilde{X}^m_s)\dd\langle \tilde{X}^m\rangle_s\notag\\
    &=\epsilon+\int_0^t\phi_{\delta\epsilon}'(\tilde{X}_s^m)K_\mu(s,s)\dd \tilde{Z}_s^m+\int_0^t\phi_{\delta\epsilon}'(\tilde{X}_s^m)\bigg(\int_0^s\partial_2 K_\mu(u,s)\dd \tilde{Z}_u^m\bigg)\dd s\notag\\
    &\qquad + \int_0^t\phi_{\delta\epsilon}'(\tilde{X}_s^m)\tilde{H}_s^m\dd s + \int_0^t\phi_{\delta\epsilon}'(\tilde{X}_s^m)K_\sigma(s,s)\dd \tilde{Y}_s^m \notag\\
    &\qquad + \frac{1}{2}\int_0^t \phi_{\delta\epsilon}''(\tilde{X}_s^m) K_\sigma(s,s)^2\Big(\sigma\big(s,X^{m+1}(\kappa_{m+1}(s))\big)-\sigma\big(s,X^m(\kappa_m(s))\big)\Big)^2 \dd s\notag\\
    &=:\epsilon+I_{1,t}^{\delta\epsilon}+I_{2,t}^{\delta\epsilon}+I_{3,t}^{\delta\epsilon}+I_{4,t}^{\delta\epsilon}+I_{5,t}^{\delta\epsilon},
  \end{align}
  for $t\in [0,T]$.

  In order to bound $\E[|\tilde{X}^m_t|]$, we shall estimate the five terms appearing in \eqref{eqmean.1} separately. We set
  \begin{equation*}
     U^m_t:=|X^m(t)-X^m(\kappa_m(t))|,\quad t\in [0,T].
  \end{equation*}

  For $I_{1,t}^{\delta\epsilon}$, we use the boundedness of $K_\mu$ (Assumption~\ref{ass:kernels}), the Lipschitz continuity of $\mu$ (Assumption~\ref{ass:coefficients}~(ii)) and the bound $\|\phi'_{\delta\epsilon}\|_\infty\leq 1$ to estimate
  \begin{align*}
    \E[I_{1,t}^{\delta\epsilon}]
    &=\E\bigg[\int_0^{t}\phi_{\delta\epsilon}'(\tilde{X}_s^m)K_\mu(s,s)\Big(\mu\big(s,X^{m+1}(\kappa_{m+1}(s))\big)-\mu\big(s,X^{m}(\kappa_m(s))\big)\Big)\dd s\bigg]\notag\\
    &\leq C \E\bigg[\int_0^{t}\big(|\tilde{X}_{s}^m|+U^m_s+ U^{m+1}_s\big)\dd s\bigg].
  \end{align*}
  Since, by Corollary~\ref{cor:bound},
  \begin{equation*}
    \E\bigg[\int_0^{t}(U^m_s+ U^{m+1}_s)\dd s\bigg]
    \leq C 2^{-\beta m^5}
  \end{equation*}
  for any $\beta \in (0,\gamma)$, we get
  \begin{align}\label{i1.1}
    \E[I_{1,t}^{\delta\epsilon}]
    &\leq C \bigg(2^{-\beta m^5}+\int_0^{t}\E\big[|\tilde{X}_{s}^m|\big]\dd s\bigg).
  \end{align}

  For $I_{2,t}^{\delta\epsilon}$, using the boundedness of $\partial_2K_\mu(u,s)$ on $\Delta_T$ (Assumption~\ref{ass:kernels}), the Lipschitz continuity of $\mu$ (Assumption~\ref{ass:coefficients}~(ii)) and the bound $\|\phi'_{\delta\epsilon}\|_\infty\leq 1$, we obtain
  \begin{align*}
    &\E[I_{2,t}^{\delta\epsilon}]\\
    &=\E\bigg[\int_0^{t}\phi_{\delta\epsilon}'(\tilde{X}_s^m)\bigg(\int_0^s \partial_2K_\mu(u,s)\Big(\mu\big(u,X^{m+1}(\kappa_{m+1}(u))\big)-\mu\big(u,X^{m}(\kappa_m(u))\big)\Big)\dd u\bigg)\dd s\bigg]\\
    & \leq C \E\bigg[\int_0^{t}\big(|\tilde{X}_{s}^m|+U^m_s+ U^{m+1}_s\big)\dd s\bigg].
  \end{align*}
  Hence, as for $I_{1,t}^{\delta \epsilon}$, we arrive at
  \begin{equation}\label{i1.1_2}
    \E[I_{2,t}^{\delta\epsilon}]
    \leq C \bigg(2^{-\beta m^{5}}+\int_0^{t}\E\big[|\tilde{X}_{s}^m|\big]\dd s\bigg).
  \end{equation}

  For $I_{3,t}^{\delta\epsilon}$, we have
  \begin{equation*}
    \E[I_{3,t}^{\delta\epsilon}]
    =\E\bigg[\int_0^{t}\phi_{\delta\epsilon}'(\tilde{X}_s^m)\tilde{H}^m_s\dd s\bigg].
  \end{equation*}
  Noting that an application of the integration by parts formula for semimartingales (cf. \cite[Theorem~(VI).38.3]{Rogers2000}) gives
  \begin{align*}
    \tilde{H}_s^m=\int_0^s \partial_2 K_\sigma(u,s)\dd \tilde{Y}_u^m=\partial_2 K_\sigma(s,s)\tilde{Y}_s^m-\int_0^s \tilde{Y}_u^m \partial_{21}K_\sigma(u,s)\dd u,
  \end{align*}
  we use $\|\phi'_{\delta\epsilon}\|_\infty\leq 1$ and the stochastic Fubini theorem to get
  \begin{align}\label{i2.1}
    \E[I_{3,t}^{\delta\epsilon}]
    &\leq \int_0^{t} \E[|\tilde{H}_{s}^m|]\dd s\notag\\
    &\leq \int_0^t |\partial_2 K_\sigma(s,s)|\E[|\tilde{Y}_{s}^m|]\dd s + \int_0^t \int_0^s |\partial_{21}K_\sigma(u,s)|\E[|\tilde{Y}_{u}^m|]\dd u \dd s\notag\\
    &\leq \int_0^t \E[|\tilde{Y}_{s}^m|]\bigg( | \partial_2 K_\sigma(s,s) | + \int_s^t |\partial_{21}K_\sigma(s,u)|\dd u \bigg)\dd s.
  \end{align}

  For $I_{4,t}^{\delta\epsilon}$, we get
  \begin{align}\label{i3.1}
    &\E[I^{\delta\epsilon}_{4,t}]\notag\\
    &\quad=\E\left[ \int_0^{t}\phi_{\delta\epsilon}'(\tilde{X}_s^m)K_\sigma(s,s)\Big(\sigma\big(s,X^{m+1}(\kappa_{m+1}(s))\big)-\sigma\big(s,X^m(\kappa_m(s))\big)\Big)\dd B_s \right]\notag\\
    &\quad=0,
  \end{align}
  since $I_{4,t}^{\delta\epsilon}$ is a martingale by \cite[p.73, Corollary 3]{Pro1992}, since $\E[\langle I_{4,t}^{\delta\epsilon} \rangle_t]<\infty$ for all $t\in[0,T]$ due to the boundedness of $K_\sigma$ (Assumption~\ref{ass:kernels}), the growth bound on $\sigma$ and Proposition~\ref{prop:Xm}.

  For $I_{5,t}^{\delta\epsilon}$, using the boundedness of $K_\sigma$ (Assumption~\ref{ass:kernels}), the H{\"o}lder continuity of~$\sigma$ (Assumption~\ref{ass:coefficients}~(ii)) and the inequality~\eqref{ineq:phi}, we get that
  \begin{align}
    \E[I_{5,t}^{\delta\epsilon}]
    &=\E\bigg[\frac{1}{2}\int_0^{t} \phi_{\delta\epsilon}''(\tilde{X}_s^m) K_\sigma(s,s)^2\Big(\sigma\big(s,X^{m+1}(\kappa_{m+1}(s))\big)-\sigma\big(s,X^m(\kappa_m(s))\big)\Big)^2 \dd s \bigg]\notag\\
    &\leq C \E\bigg[\int_0^{t} \phi_{\delta\epsilon}''(\tilde{X}_{s}^m)\big(|\tilde{X}^m({s})|+U^m_s+U^{m+1}_s\big)^{1+2\xi}\dd s\bigg]\notag\\
    &\leq C \E\bigg[\int_0^{t} \1_{[\frac{\epsilon}{\delta},\epsilon]}(|\tilde{X}^m({s})|) \frac{\big(|\tilde{X}^m({s})|+U^m_s+U^{m+1}_s\big)^{1+2\xi}}{|\tilde{X}^m({s})|\ln(\delta)}\dd s\bigg]\notag\\
    &\leq C\bigg( \frac{\epsilon^{2\xi}}{\ln(\delta)}+ \frac{\delta}{\epsilon\ln(\delta)}\E\bigg[ \int_0^{t}\big( U^m_s+ U^{m+1}_s\big)^{1+2\xi}\dd s\bigg]\bigg).\label{eq:YW_notenough}
  \end{align}
  Moreover, by Corollary~\ref{cor:bound}, we derive that
  \begin{equation*}
    \E\bigg[ \int_0^{t}\big( U^m_s+ U^{m+1}_s\big)^{1+2\xi}\dd s\bigg]
    \leq C 2^{-(1+2\xi)\beta m^{5}}
  \end{equation*}
  for any $\beta\in (0,\gamma)$ and, hence, we conclude
  \begin{equation}\label{ineq:i4}
    \E[I_{5,t}^{\delta\epsilon}]
    \leq C\bigg( \frac{\epsilon^{2\xi}}{\ln(\delta)}+ \frac{\delta}{\epsilon\ln(\delta)}2^{-(1+2\xi)\beta m^{5}}\bigg).
  \end{equation}

  Combining \eqref{eqmean.1} with the five estimates \eqref{i1.1}, \eqref{i1.1_2}, \eqref{i2.1}, \eqref{i3.1} and \eqref{ineq:i4}, we end up with
  \begin{align*}
    \E[|\tilde{X}^m_{t}|]
    &\leq C\bigg(2^{-\beta m^5} + \frac{\epsilon^{2\xi}}{\ln(\delta)}+ \frac{\delta}{\epsilon\ln(\delta)}2^{-(1+2\xi)\beta m^5} + \int_0^t \E[|\tilde{X}^m_{s}|]\dd s\\
    &\qquad + \int_0^t \E[|\tilde{Y}_{s}^m|] \bigg(  |\partial_2 K_\sigma(s,s) | + \int_s^{t} |\partial_{21} K_\sigma(s,u)|\dd u \bigg)\dd s \bigg).
  \end{align*}
  Therefore, choosing $\delta :=2^{\rho m^5}$ for $\rho \in (0,((1+2\xi)\beta)/2]$ and $\epsilon := 2^{-\frac{(1+2\xi)\beta}{2} m^5}$,
  we get
  \begin{align}\label{part1.1}
    \E[|\tilde{X}^m_{t}|]
    &\leq C\bigg(C_m + \int_0^t \E[|\tilde{X}^m_{s}|]\dd s\notag\\
    &\qquad + \int_0^t \E[|\tilde{Y}_{s}^m|] \bigg( |\partial_2 K_\sigma(s,s)| + \int_s^{t} |\partial_{21} K_\sigma(s,u)|\dd u \bigg)\dd s \bigg),
  \end{align}
  with
  \begin{equation}\label{eq:C_m}
    C_m := 2^{-\beta m^5} + m^{-5} 2^{-(1+2\xi)\beta \xi m^5}
    +m^{-5} 2^{-(\frac{(1+2\xi)\beta}{2}-\rho)m^5}.
  \end{equation}

  To apply a Gr{\"o}nwall lemma, we set
  \begin{equation*}
    M_{m}(t):=\sup\limits_{s\in [0,t]}\left(\E[|\tilde{X}^m_{s}|] + \E[|\tilde{Y}^m_{s}|]\right),\quad t\in [0,T],
  \end{equation*}
  and derive in the following an inequality of the form $M_{m}(t) \leq C_m + \int_0^t f(t-s)M_{m} (s) \dd s$ for a suitable function~$f$.

  To get a bound for $\E[|\tilde{Y}^m_{t}|]$, we first apply the integration by part formula to obtain
  \begin{align*}
    \tilde{X}^m_t &=\int_0^tK_\mu(s,t)\Big(\mu\big(s,X^{m+1}(\kappa_{m+1}(s))\big)-\mu\big(s,X_s^m(\kappa_m(s))\big)\Big)\dd s + \int_0^t K_\sigma(s,t)\dd \tilde{Y}^m_s\notag\\
    &= \int_0^t K_\mu(s,t)\Big(\mu\big(s,X^{m+1}(\kappa_{m+1}(s))\big)-\mu\big(s,X_s^m(\kappa_m(s))\big)\Big)\dd s\notag\\
    &\qquad + K_\sigma(t,t)\tilde{Y}^m_t-\int_0^t\partial_1 K_\sigma (s,t)\tilde{Y}^m_s\dd s,
  \end{align*}
  where we used that $K_\sigma(\cdot,t)$ is absolutely continuous for every $t\in [0,T]$. Since $K_\sigma(t,t)>C$ for some constant~$C>0$, $K_\mu$ is bounded (both by Assumption~\ref{ass:kernels}) and $\mu$ is Lipschitz continuous (Assumption~\ref{ass:coefficients}), we get
  \begin{align*}
    \E [|\tilde{Y}^m_{t}|]
    \leq & C \E \bigg [|\tilde{X}^m_{t}|
    +\int_0^{t} |K_\mu(s,t)|\Big|\mu\big(s,X^{m+1}(\kappa_{m+1}(s))\big)
    - \mu\big(s,X_s^m(\kappa_m(s))\big)\Big|\dd s \\
    &\qquad +\int_0^{t}|\partial_1 K_\sigma (s,t)||\tilde{Y}^m_s|\dd s\bigg] \\
    \leq & C \bigg(\E[|\tilde{X}^m_{t}|]
    + \int_0^{t} \E\big[|\tilde{X}^m_{s}\big|]\dd s
    +\E\bigg[\int_0^{t} (U^m_s+U^{m+1}_s) \dd s\bigg] \\
    &\qquad+\int_0^{t}|\partial_1 K_\sigma (s,t)|\E[|\tilde{Y}^m_{s}|]\dd s \bigg) \\
    \leq & C \bigg(2^{-\beta m^5}+\E[|\tilde{X}^m_{t}|]
    + \int_0^{t} \E\big[|\tilde{X}^m_{s}\big|]\dd s
    +\int_0^{t}|\partial_1 K_\sigma (s,t)|\E[|\tilde{Y}^m_{s}|]\dd s \bigg),
  \end{align*}
  where we used Corollary~\ref{cor:bound} for the last estimate. Hence, by \eqref{part1.1} we obtain
  \begin{align}\label{part2i.1}
    \E [|\tilde{Y}^m_{t}|]
    &\leq  C \bigg(C_m + \int_0^t \E[|\tilde{X}^m_{s}|]\dd s\notag\\
    &\quad + \int_0^t \E[|\tilde{Y}_{s}^m|] \bigg( |\partial_1 K_\sigma (s,t)|+ | \partial_2 K_\sigma(s,s)| + \int_s^{t} |\partial_{21} K_\sigma(s,u)|\dd u \bigg)\dd s \bigg).
  \end{align}

  By the bound on the partial derivatives of $K_\sigma$ made in Assumption~\ref{ass:kernels},  \eqref{part1.1} and \eqref{part2i.1} can be further estimated to
  \begin{align*}
    &\E[|\tilde{X}^m_{t}|]
    \leq C\bigg(C_m +\int_0^t\E[|\tilde{X}^m_{s}|]\dd s + \int_0^t (t-s)^{-\alpha}\, \E[|\tilde{Y}^m_{s}|]\dd s \bigg),\\
    &\E[|\tilde{Y}^m_{t}|]
    \leq C\bigg(C_m + \int_0^t\E[|\tilde{X}_{s}^m|]\dd s + \int_0^t (t-s)^{-\alpha}\, \E[|\tilde{Y}^m_{s}|]\dd s \bigg),
  \end{align*}
  for $\alpha\in [0,\frac{1}{2})$ as given in Assumption~\ref{ass:kernels}. Hence, we arrive at
  \begin{align*}
    M_{m}(t)&\leq \sup\limits_{s\in [0,t]}\E[|\tilde{X}^m_{t}|] + \sup\limits_{s\in [0,t]}\E[|\tilde{Y}^m_{t}|] \notag\\
    &\leq C \bigg(C_m +\int_0^t \big(1+(t-s)^{-\alpha}\big) M_{m}(s)\dd s\bigg).
  \end{align*}
  Note that Proposition~\ref{prop:Xm} secures the integrability of $M_m$. An application of the Gr{\"o}nwall's lemma for weak singularities (see e.g. \cite[Lemma~A.2]{Kruse2014}) reveals that $M_{m}(t)\leq C C_m$. The claimed summability of the sequence $(C_m)_{m\in\N}$ follows immediately by \eqref{eq:C_m}.
\end{proof}

\begin{remark}\label{rem:approximation}
  The approximation $\phi_{\delta\epsilon}$ of the absolute value, as used in the proof of Theorem~\ref{thm:existence}, was introduced by Gy{\"o}ngy and R{\'a}sonyi \cite{Gyongy2011}. It is a modification of the approximation originally used by Yamada and Watanabe \cite{Yamada1971} and appears to be more involved. While the original approximation of Yamada and Watanabe is sufficient to prove pathwise uniqueness, as we will also see in Section~\ref{sec:pathwise uniqueness}, to prove the existence of a solution the approximation $\phi_{\delta\epsilon}$ seems necessary. Indeed, one needs $\epsilon \to 0$ to ensure that $\phi_{\delta\epsilon} \to |\cdot|$ but the second parameter $\delta$ is essential to obtain the convergence of the Euler type approximation~$(X^m)_{m\in \N}$ in the case $\xi=0$ (i.e. $\sigma$ is $1/2$-H{\"o}lder continuous), as one can see from~\eqref{part1.1} and~\eqref{eq:C_m},
\end{remark}

With these preparation at hand we are ready to prove Theorem~\ref{thm:existence}.

\begin{proof}[Proof of Theorem~\ref{thm:existence}]
  \textit{Step~1:} The sequence $(X^m)_{m\in\N}$ is a Cauchy sequence in $L^p(\Omega\times [0,T])$ for $p$ given in the statement of Theorem~\ref{thm:existence}.
  
  By Fubini's theorem and Lemma~\ref{lem:1}, there exists a sequence $(C_m)_{m\in \N}$ such that
  \begin{equation*}
    \E\bigg[\int_0^T \left|X^{m+1}(s)-X^m(s)\right| \dd s\bigg]
    \leq C \sup\limits_{ s\in [0,T]} \E\left[|X^{m+1}(s)-X^m(s)|\right]
    \leq C_m
  \end{equation*}
  for $m\in \N$. Hence, using H{\"o}lder's inequality and the moment bound for $(X^m(t))_{t\in[0,T]}$ from Proposition~\ref{prop:Xm}, we get
  \begin{align*}
    &\E\bigg[\int_0^T |X^{m+1}(t)-X^m(t)|^p \dd t\bigg]\\
    &\qquad\leq \E\bigg[\int_0^T |X^{m+1}(t)-X^m(t)|^{2p-1}\dd t\bigg]^{\frac{1}{2}}
    \E\bigg[\int_0^T |X^{m+1}(t)-X^m(t)|\dd t\bigg]^{\frac{1}{2}}\\
    &\qquad\leq 2^{p-1}\bigg(1+\sup_{t\in [0,T]} |x_0(t)|^{2p-1}) \bigg)^{\frac{1}{2}}C_m^{\frac{1}{2}}.
  \end{align*}
  Due to the summability property of $(C_m)_{m\in\N}$, the sequence $(X^m)_{m\in\N}$ is a Cauchy sequence in $L^p(\Omega\times [0,T])$. Hence, there exists a process $X=(X_t)_{t\in [0,T]}\in L^p(\Omega\times [0,T])$, such that
  \begin{equation}\label{eq:convX}
    \lim\limits_{m\to\infty}\E\left[\int_0^T \left|X^{m}(s)-X_s\right|^p \dd s\right]=0.
  \end{equation}

  \textit{Step~2:} $(X_t)_{t\in [0,T]}$ yields a strong solution to the SVE~\eqref{eq:SVE}

  By construction, the processes $(X^m(t))_{t\in [0,T]}$ are $(\mathcal{F}_t)_{t\in[0,T]}$-progressively measurable on the given probability space $(\Omega,\mathcal{F},(\mathcal{F}_t)_{t\in[0,T]},\mathbb{P})$. Since \eqref{eq:convX} also shows the $L^p([0,t]\times\Omega)$-convergence of $(X^m_s)_{s\in[0,t]}$ to $(X_s)_{s\in[0,t]}$ for every $t\in[0,T]$, the completeness of the $L^p$ spaces (see e.g. \cite[Theorem~7.3]{Klenke2014}) yields $\mathcal{B}([0,t])\otimes\mathcal{F}_t$-measurability of $(s,\omega)\mapsto X_s(\omega)$, $(s,\omega)\in [0,t]\times\Omega$ for every $t\in[0,T]$. Hence, the process $(X_t)_{t\in[0,T]}$ is also $(\mathcal{F}_t)_{t\in[0,T]}$-progressively measurable on $(\Omega,\mathcal{F},(\mathcal{F}_t)_{t\in[0,T]},\mathbb{P})$. Moreover, by the growth conditions on $\mu$ and $\sigma$ (see Assumption~\ref{ass:coefficients}~(i)) and the integrability properties of $K_{\mu}$ and $K_{\sigma}$, we get that
  \begin{equation*}
    \int_0^t (|K_\mu(s,t)\mu(s,X_s)|+|K_\sigma(s,t)\sigma(s,X_s)|^2 )\dd s<\infty \quad \text{for all} t\in[0,T].
  \end{equation*}

  It remains to show that the process $(X_t)_{t\in [0,T]}$ fulfills the SVE~\eqref{eq:SVE}. To that end, we show that the two integrals in \eqref{eq:X_n} preserve the $L^p(\Omega\times [0,T])$-convergence. For the Riemann--Stieltjes integral, we use the boundedness of $K_\mu$, the Lipschitz continuity of $\mu$, H{\"o}lder's inequality and Fubini's theorem to obtain
  \begin{align*}
    &\E\left[ \int_0^T\Big| \int_0^t K_\mu(s,t)\left( \mu(s,X^m(\kappa_m(s)))-\mu(s,X_s) \right)\dd s \Big|^p \dd t \right] \\
    &\quad\quad \leq C \int_0^T\int_0^T \E\left[|X^m(\kappa_m(s))-X_s|^p\right]\dd s\dd t \\
    &\quad\quad \leq C \bigg( \E\bigg[ \int_0^T |X^m(\kappa_m(s))-X^m(s)|^p\dd s\bigg] + \E\bigg[ \int_0^T |X^m(s)-X_s|^p\dd s\bigg] \bigg)\to 0,
  \end{align*}
  as $m\to\infty$ by Corollary~\ref{cor:bound} and \eqref{eq:convX}. For the stochastic integral, we use Fubini's theorem, Burkholder--Davis--Gundy's inequality, H{\"o}lder's inequality, the boundedness of~$K_\sigma$, and the H{\"o}lder regularity of~$\sigma$ to get that
  \begin{align*}
    &\E\left[ \int_0^{T} \left| \int_0^t K_\sigma(s,t)\left( \sigma(s,X^m(\kappa_m(s)))-\sigma(s,X_s) \right)\dd B_s \right|^p\dd t \right]\\
    &\quad\qquad = \int_0^{T} \E\bigg[\big| \int_0^t K_\sigma(s,t)\left( \sigma(s,X^m(\kappa_m(s)))-\sigma(s,X_s) \right)\dd B_s \big|^{p}\bigg]\dd t \\
    &\quad\qquad\leq \int_0^{T} \E\bigg[ \int_0^t K_\sigma(s,t)^2\left( \sigma(s,X^m(\kappa_m(s)))-\sigma(s,X_s) \right)^2\dd s\bigg]^{\frac{p}{2}}\dd t \\
    &\quad\qquad\leq C\left( \int_0^{T} \int_0^T \E[|X^m(\kappa_m(s))-X_s|^{\frac{p}{2}+p\xi}] \dd s\dd t \right)\\
    &\quad\qquad\leq C\bigg(\E\bigg[\int_0^T |X^m(\kappa_m(s))-X^m(s)|^{\frac{p}{2}+p\xi} \dd s\bigg]+\E\bigg[\int_0^T |X^m(s)-X_s|^{\frac{p}{2}+p\xi} \dd s\bigg]\bigg).
  \end{align*}
  Thus, by Corollary~\ref{cor:bound} and the convergence $X^m\to X$ in $L^{\frac{p}{2}+p\xi}(\Omega\times [0,T])$ as $m\to\infty$, for $\xi\in [0,\frac{1}{2}]$, which is implied by the one in $L^p(\Omega\times [0,T])$, we see that the stochastic integral does preserve the $L^p(\Omega\times [0,T])$-convergence. Thus, we have proven that the limiting process $(X_t)_{t\in [0,T]}$ fulfills the SVE~\eqref{eq:SVE} for almost all $(t,\omega)\in [0,T]\times\Omega$. By Remark~\ref{rem_Kolmogorov}, $(X_t)_{t\in [0,T]}$ has an $\mathbb{P}$-a.s. continuous version, which fulfills the SVE~\eqref{eq:SVE} for all $t\in[0,T]$ for almost all $\omega\in\Omega$, and hence, is a strong solution of \eqref{eq:SVE}.
\end{proof}

\section{Pathwise uniqueness}\label{sec:pathwise uniqueness}

In this section we establish the pathwise uniqueness for the stochastic Volterra equation~\eqref{eq:SVE} under Assumptions~\ref{ass:kernels}, \ref{ass:coefficients}~(i), and under slightly weaker regularity assumptions on $\mu$ and $\sigma$ than Assumption~\ref{ass:coefficients}~(ii), namely an Osgood-type condition on~$\mu$ and the Yamada--Watanabe condition on~$\sigma$, as formulated in the next assumption.

\begin{assumption}\label{ass:Osg_YW}
  Let $\mu,\sigma\colon [0,T]\times\R\to\R$ be measurable functions such that:
  \begin{itemize}
    \item[(i)] there is some continuous, non-decreasing and concave function $\kappa\colon [0,\infty)\to [0,\infty)$ with $\kappa(0)=0$ and $\kappa(x)>0$ for $x>0$, such that, with the notation $\tilde{\kappa}(x):=\kappa(x)+|x|$,
    \begin{equation*}
      \int_0^\epsilon \frac{\dd x}{\big( \tilde{\kappa}(\sqrt[q]{x}) \big)^q}=\infty,
    \end{equation*}
    holds for all $\epsilon>0$ and $q\in(\frac{1}{1-\alpha},\frac{1}{1-\alpha}+\tilde{\epsilon})$ for some $\tilde{\epsilon}>0$, where $\alpha\in[0,\frac{1}{2})$ is given by Assumption~\ref{ass:kernels}~(ii), and
    \begin{equation*}
      |\mu(t,x)-\mu(t,y)| \leq \kappa(|x-y|),
    \end{equation*}
    for all $t\in[0,T]$, $x,y\in\R$,
    \item[(ii)] there is some continuous strictly increasing function $\rho\colon [0,\infty)\to [0,\infty)$ with $\rho(0)=0$ and $\rho(x)>0$ for $x>0$, such that
    \begin{equation*}
      \int_0^\epsilon \frac{\dd x}{\rho(x)^2}=\infty,
    \end{equation*}
    holds for all $\epsilon>0$, and
    \begin{equation*}
      |\sigma(t,x)-\sigma(t,y)| \leq \rho(|x-y|),
    \end{equation*}
    for all $t\in[0,T]$, $x,y\in\R$.
  \end{itemize}
\end{assumption}

\begin{remark}
  Choosing $\kappa(x)=C_\mu|x|$ and $\rho(x)=C_\sigma|x|^{\frac{1}{2}+\xi}$ shows that Assumption~\ref{ass:coefficients}~(ii) implies Assumption~\ref{ass:Osg_YW}. We note that if $\mu$ is assumed to be Lipschitz continuous and $\sigma$ to fulfill the Yamada--Watanabe condition, it is sufficient to use a fractional Gr{\"o}nwall lemma like the one in \cite[Lemma~A.2]{Kruse2014} instead of the fractional Bihari inequality in \eqref{ineq:Bihari}. Moreover, if one considers $K_\sigma=1$, the Osgood-type condition in Assumption~\ref{ass:Osg_YW}~(i) can be replaced by the classical Osgood condition for SDEs (see e.g. \cite[Chapter~5, Remark~2.16]{Karatzas1991}) since one can then use the classical instead of the fractional Bihari inequality and the application of integration by parts to the stochastic integral is not required.
\end{remark}

The main result of this section reads as follows.

\begin{theorem}\label{thm:uniqueness}
  Suppose Assumptions~\ref{ass:kernels}, \ref{ass:coefficients}~(i) and \ref{ass:Osg_YW}. Then, pathwise uniqueness holds for the stochastic Volterra equation~\eqref{eq:SVE}.
\end{theorem}

\begin{proof}
  Since the proof relies partly on similar techniques as the proof of Lemma~\ref{lem:1}, we try to give a condense presentation and refer to the analogue calculation in Section~\ref{sec:existence}.

  Let $(X_t^1)_{t\in[0,T]}$ and $(X_t^2)_{t\in[0,T]}$ be solutions to the SVE~\eqref{eq:SVE}. Analogously to Section~\ref{sec:existence}, we define  $Y_t^i:=\int_0^t \sigma(s,X_s^i)\dd B_s$ and $H_t^i:=\int_0^t\partial_2 K_\sigma(s,t)\dd Y_s^i$, for $i=1,2$, as well as  $\tilde{Y}_t:=Y_t^1-Y_t^2$, $\tilde{X}_t:=X_t^1-X_t^2$, $\tilde{H}_t:=H_t^1-H_t^2$, and $\tilde{Z}_t:=\int_0^t\big(\mu(s,X^1_s)-\mu(s,X^2_s)\big)\dd s$, for $t\in [0,T]$. By Lemma~\ref{lem:semimartingale property}, we obtain the semimartingale decomposition
  \begin{align}\label{semmart}
    \tilde{X}_t&=\int_0^tK_\mu(s,s)(\mu(s,X_s^1)-\mu(s,X_s^2))\dd s +\int_0^t\int_0^s \partial_2 K_\mu(u,s)\dd \tilde{Z}_u\dd s \nonumber\\
    &\qquad +\int_0^t\tilde{H}_s\dd s+\int_0^tK_\sigma(s,s)\dd \tilde{Y}_s,\quad t\in[0,T].
  \end{align}

  To construct an approximation of the absolute value by smooth functions allowing us to apply It{\^o}'s formula, we use the classical approximation of Yamada--Watanabe~\cite{Yamada1971} for simplicity, cf. Remark~\ref{rem:approximation}. Based on the strictly increasing function $\rho$ from Assumption~\ref{ass:Osg_YW}~(ii), we define a sequence $(\phi_n)_{n\in\N}$ of functions mapping from $\R$ to $\R$ that approximates the absolute value in the following way: Let $(a_n)_{n\in\N}$ be a strictly decreasing sequence with $a_0=1$ such that $a_n\to 0$ as $n\to \infty$ and
  \begin{align*}
    \int_{a_n}^{a_{n-1}}\frac{1}{\rho(x)^2}\dd x=n.
  \end{align*}
  Furthermore, we define a sequence of mollifiers: let $(\psi_n)_{n\in\N}\in C_0^{\infty}(\R)$ be smooth functions with compact support such that $\textup{supp}(\psi_n)\subset (a_n,a_{n-1})$, and with the properties
  \begin{align}\label{prop}
    0\leq \psi_n(x)\leq \frac{2}{n\rho(x)^2}, \quad \forall x\in\R,
    \quad\text{and}\quad
    \int_{a_n}^{a_{n-1}}\psi_n(x)\dd x=1.
  \end{align}
  We set
  \begin{align*}
    \phi_n(x):=\int_0^{|x|}\left(\int_0^y \psi_n(z)\dd z\right)\dd y,\quad x\in \R.
  \end{align*}
  By \eqref{prop} and the compact support of $\psi_n$, it follows that $\phi_n(\cdot)\to |\cdot|$ uniformly as $n\to \infty$. Since every $\psi_n$ and, thus, every $\phi_n$ is zero in a neighborhood around zero, the functions~$\phi_n$ are smooth with
  \begin{equation*}
    \|\phi_n'\|_\infty\leq 1,
    \quad
    \phi_n'(x)=\textup{sgn}(x)\int_0^{|x|}\psi_n(y)\dd y,
    \quad\text{and}\quad
    \phi_n''(x)=\psi_n(|x|)
    \quad\text{for } x\in\R.
  \end{equation*}

  Since the quadratic variation of the semimartingale $(\tilde{X}_t)_{t\in[0,T]}$ is given by
  \begin{align*}
    \langle\tilde{X}\rangle_t=\int_0^t K_\sigma(s,s)^2\big(\sigma(s,X_s^1)-\sigma(s,X_s^2)\big)^2\dd s,
    \quad t\in [0,T],
  \end{align*}
  we get, by applying It{\^o}'s formula and using the semimartingale decomposition \eqref{semmart}, that
  \begin{align}\label{eqmean}
    \phi_n(\tilde{X}_t)
    =&\int_0^t\phi_n'(\tilde{X}_s)\dd\tilde{X}_s+\frac{1}{2}\int_0^t \phi_n''(\tilde{X}_s)\dd\langle \tilde{X}\rangle_s\notag\\
    =&\int_0^t\phi_n'(\tilde{X}_s)K_\mu(s,s)(\mu(s,X_s^1)-\mu(s,X_s^2))\dd s + \int_0^t \phi_n'(\tilde{X}_s)\bigg(\int_0^s \partial_2 K_\mu(u,s)\dd \tilde{Z}_u\bigg)\dd s \notag\\
    &+ \int_0^t\phi_n'(\tilde{X}_s)\tilde{H}_s\dd s
    + \int_0^t\phi_n'(\tilde{X}_s)K_\sigma(s,s)\dd \tilde{Y}_s\notag\\
     &+ \frac{1}{2}\int_0^t \phi_n''(\tilde{X}_s) K_\sigma(s,s)^2\big(\sigma(s,X_s^1)-\sigma(s,X_s^2)\big)^2 \dd s\notag\\
    =:&I_{1,t}^n+I_{2,t}^n+I_{3,t}^n+I_{4,t}^n+I_{5,t}^n
  \end{align}
  for $t\in [0,T]$.

  For $I_{1,t}^n$, we use Assumption~\ref{ass:Osg_YW}~(i), the boundedness of~$K_\mu$ (Assumption~\ref{ass:kernels}), the bound $\|\phi'_n\|_\infty~\leq~1$ and Jensen's inequality to estimate
  \begin{equation}\label{i1}
    \E[I_{1,t}^n]\leq C \int_0^t \E[\kappa(|\tilde{X}_{s}|)]\dd s\leq C \int_0^t \kappa(\E[|\tilde{X}_{s}|])\dd s.
  \end{equation}
  For $I_{2,t}^n$, we additionally use the boundedness of $\partial_2 K_\mu(u,s)$ on $\Delta_T$ to obtain
  \begin{equation}\label{i12}
    \E[I_{2,t}^n]\leq C \int_0^t\kappa(\E[|\tilde{X}_{s}|])\dd s.
  \end{equation}
  For $I_{3,t}^n$, similarly to \eqref{i2.1}, we use the integration by parts formula to estimate
  \begin{align}\label{i2}
    \E[I_{3,t}^n]&\leq \int_0^t \E[|\tilde{H}_{s}|]\dd s\notag\\
    &\leq \int_0^t |\partial_2 K_\sigma(s,s)|\E[|\tilde{Y}_{s}|]\dd s + \int_0^t \int_0^s |\partial_{21}K_\sigma(u,s)|\E[|\tilde{Y}_{u}|]\dd u \dd s\notag\\
    &\leq \int_0^t \E[|\tilde{Y}_{s}|] \bigg( \partial_2 K_\sigma(s,s) + \int_s^t |\partial_{21}K_\sigma(s,u)|\dd u \bigg)\dd s.
  \end{align}
  For $I_{4,t}^n$, since $I_{4,t}^n$ is a martingale by \cite[p.73, Corollary~3]{Pro1992} due to the boundedness of $K_\sigma$, the growth bound on $\sigma$ and Lemma~\ref{lem:bound}, we get
  \begin{align}
    \E[I^n_{4,t}]=\E\left[ \int_0^{t}\phi_n'(\tilde{X}_s)K_\sigma(s,s)(\sigma(s,X_s^1)-\sigma(s,X_s^2))\dd B_s \right]=0,\label{i3}
  \end{align}
  For $I_{5,t}^n$, we get by using the boundedness of $K_\sigma$ (Assumption~\ref{ass:kernels}), the regularity of $\sigma$ from Assumption~\ref{ass:Osg_YW}~(ii), and the inequality~\eqref{prop} that
  \begin{align}\label{i4}
    \E[I_{5,t}^n]&\leq C \E\bigg[\int_0^t \phi_n''(\tilde{X}_{s})\rho(|\tilde{X}_{s}|)^2\dd s\bigg] \notag\\
    &\leq C \E\bigg[\int_0^t \frac{2}{n\rho(|\tilde{X}_{s}|)^2}\rho(|\tilde{X}_{s}|)^2\dd s\bigg]\notag\\
    &\leq \frac{C}{n},
  \end{align}
  for some $C>0$.

  Finally, sending $n\to\infty$ and combining the five previous estimates \eqref{i1}, \eqref{i12}, \eqref{i2}, \eqref{i3} and \eqref{i4} with \eqref{eqmean} implies
  \begin{equation}\label{part1}
    \E[|\tilde{X}_{t}|]\leq  C \int_0^t\kappa(\E[|\tilde{X}_{s}|])\dd s
    + \int_0^t \E[|\tilde{Y}_{s}|] \bigg( \partial_2 K_\sigma(s,s) + \int_s^t |\partial_{21}K_\sigma(s,u)|\dd u \bigg)\dd s.
  \end{equation}

  To apply a Gr{\"o}nwall lemma, we set
  \begin{equation*}
    M(t):=\sup\limits_{s\in [0,t]}\left(\E[|\tilde{X}_{s}|] + \E[|\tilde{Y}_{s}|]\right),\quad t\in [0,T],
  \end{equation*}
  and derive in the following an inequality of the form $M(t)\leq \int_0^t f(t-s)\tilde{\kappa}(M(s))\dd s$ for suitable functions~$f$ and $\tilde{\kappa}$. To find a bound for $\E[|\tilde{Y}_{t}|]$, we apply the integration by part formula to obtain
  \begin{align}\label{equiv}
    \tilde{X}_t&=\int_0^t K_\mu(s,t)(\mu(s,X_s^1)-\mu(s,X_s^2))\dd s + \int_0^t K_\sigma(s,t)\dd \tilde{Y}_s\notag\\
    &= \int_0^t K_\mu(s,t)(\mu(s,X_s^1)-\mu(s,X_s^2))\dd s + K_\sigma(t,t)\tilde{Y}_t-\int_0^t\partial_1 K_\sigma (s,t)\tilde{Y}_s\dd s
  \end{align}
  keeping in mind that that $K_\sigma(\cdot,t)$ is absolutely continuous for every $t\in [0,T]$. Due to $|K_\sigma(t,t)|> C$ for some constant~$C>0$, we can rearrange \eqref{equiv} and use \eqref{part1} to get
  \begin{align}\label{part2i}
    \E\left[|\tilde{Y}_{t}|\right]
    \leq & C\bigg(
    \int_0^t\E\big[|\mu(s,X_{s}^1)-\mu(s,X_{s}^2)|\big]\dd s \notag\\
    &+ \E\big[|\tilde{X}_{t}|\big] + \int_0^t |\partial_1 K_\sigma(s,t)|\E\big[|\tilde{Y}_{s}|\big]\dd s\bigg)\notag\\
    \leq & C\bigg( \int_0^t\Big(\E\big[|\tilde{X}_{s}|\big]+\kappa(\E\big[|\tilde{X}_{s}|\big])\Big)\dd s\notag\\
    &+ \int_0^t \E\big[|\tilde{Y}_{s}|\big] \bigg( \left|\partial_1 K_\sigma(s,t)\right| + \left|\partial_2 K_\sigma(s,s)\right| + \int_s^t\left|\partial_{21}K_\sigma(s,u)\right|\dd u\bigg) \dd s\bigg).
  \end{align}
  Using Assumption~\ref{ass:kernels} to bound the partial derivative terms in \eqref{part1} and \eqref{part2i}, we end up with
  \begin{align}\label{ineq:Bihari}
    M(t)&\leq \sup\limits_{s\in [0,t]}\E[|\tilde{X}_{t}|] + \sup\limits_{s\in [0,t]}\E[|\tilde{Y}_{t}|] \notag\\
    &\leq C\bigg( \int_0^t \Big(\sup\limits_{u \in [0,s]} \E[|\tilde{X}_{u}|]+\kappa\Big(\sup\limits_{u \in [0,s]} \E[|\tilde{X}_{u}|]\Big)\Big)\dd s + \int_0^t (t-s)^{-\alpha} \sup\limits_{u \in [0,s]}\E[|\tilde{Y}_{u}|]\dd s \bigg)\notag\\
    &\leq C \int_0^t(t-s)^{-\alpha}\tilde{\kappa}(M(s))\dd s,
  \end{align}
  where $\tilde{\kappa}(x):=\kappa(x)+|x|$. An application of the fractional Bihari inequality, \cite[Theorem~2.3]{Ouaddah}, with sending $q\to \frac{1}{1-\alpha}$ like in \cite[proof of Theorem~3.1, Step~1]{Ouaddah} with the condition on $\tilde{\kappa}$ in Assumption~\ref{ass:Osg_YW}~(i) that $M(t)=0$ holds. Hence, $\tilde{X}_t=0$ almost surely, and, thus, by the continuity of the solutions, the processes~$(X^1_t)_{t\in [0,T]}$ and $(X^2_t)_{t\in [0,T]}$ are indistinguishable.
\end{proof}

\bibliography{literature}{}

\providecommand{\bysame}{\leavevmode\hbox to3em{\hrulefill}\thinspace}
\providecommand{\MR}{\relax\ifhmode\unskip\space\fi MR }
\providecommand{\MRhref}[2]{%
  \href{http://www.ams.org/mathscinet-getitem?mr=#1}{#2}
}
\providecommand{\href}[2]{#2}
\begin{thebibliography}{AJCLP21}

\bibitem[AJ21]{AbiJaber2021b}
Eduardo Abi~Jaber, \emph{Weak existence and uniqueness for affine stochastic
  {V}olterra equations with {$L^1$}-kernels}, Bernoulli \textbf{27} (2021),
  no.~3, 1583--1615.

\bibitem[AJCLP21]{AbiJaber2021}
Eduardo Abi~Jaber, Christa Cuchiero, Martin Larsson, and Sergio Pulido, \emph{A
  weak solution theory for stochastic {V}olterra equations of convolution
  type}, Ann. Appl. Probab. \textbf{31} (2021), no.~6, 2924--2952.

\bibitem[AJEE19a]{AbiJaberElEuch2019}
Eduardo Abi~Jaber and Omar El~Euch, \emph{Markovian structure of the {V}olterra
  {H}eston model}, Statist. Probab. Lett. \textbf{149} (2019), 63--72.

\bibitem[AJEE19b]{AbiJaberElEuch2019b}
\bysame, \emph{Multifactor approximation of rough volatility models}, SIAM J.
  Financial Math. \textbf{10} (2019), no.~2, 309--349.

\bibitem[AJLP19]{AbiJaber2019}
Eduardo Abi~Jaber, Martin Larsson, and Sergio Pulido, \emph{Affine {V}olterra
  processes}, Ann. Appl. Probab. \textbf{29} (2019), no.~5, 3155--3200.

\bibitem[AN97]{Alos1997}
Elisa Al{\`o}s and David Nualart, \emph{Anticipating stochastic {V}olterra
  equations}, Stochastic Process. Appl. \textbf{72} (1997), no.~1, 73--95.

\bibitem[BFG16]{Bayer2016}
Christian Bayer, Peter Friz, and Jim Gatheral, \emph{Pricing under rough
  volatility}, Quant. Finance \textbf{16} (2016), no.~6, 887--904.

\bibitem[BM80a]{Berger1980a}
Marc~A. Berger and Victor~J. Mizel, \emph{Volterra equations with {I}t\^{o}
  integrals. {I}}, J. Integral Equations \textbf{2} (1980), no.~3, 187--245.

\bibitem[BM80b]{Berger1980b}
\bysame, \emph{Volterra equations with {I}t\^{o} integrals. {II}}, J. Integral
  Equations \textbf{2} (1980), no.~4, 319--337.

\bibitem[CD01]{Coutin2001}
L.~Coutin and L.~Decreusefond, \emph{Stochastic {V}olterra equations with
  singular kernels}, Stochastic analysis and mathematical physics, Progr.
  Probab., vol.~50, Birkh\"auser Boston, Boston, MA, 2001, pp.~39--50.

\bibitem[CLP95]{Cochran1995}
W.~George Cochran, Jung-Soon Lee, and J\"urgen Potthoff, \emph{Stochastic
  {V}olterra equations with singular kernels}, Stochastic Process. Appl.
  \textbf{56} (1995), no.~2, 337--349.

\bibitem[CT20]{Cuchiero2020}
Christa Cuchiero and Josef Teichmann, \emph{Generalized {F}eller processes and
  {M}arkovian lifts of stochastic {V}olterra processes: the affine case}, J.
  Evol. Equ. \textbf{20} (2020), no.~4, 1301--1348.

\bibitem[EER19]{ElEuch2019}
Omar El~Euch and Mathieu Rosenbaum, \emph{The characteristic function of rough
  {H}eston models}, Math. Finance \textbf{29} (2019), no.~1, 3--38.

\bibitem[GR11]{Gyongy2011}
Istv{\'a}n Gy{\"o}ngy and Mikl{\'o}s R{\'a}sonyi, \emph{A note on {E}uler
  approximations for {SDE}s with {H}\"{o}lder continuous diffusion
  coefficients}, Stochastic Process. Appl. \textbf{121} (2011), no.~10,
  2189--2200.

\bibitem[GRR71]{Garsia1970}
A.~M. Garsia, E.~Rodemich, and H.~Rumsey, Jr., \emph{A real variable lemma and
  the continuity of paths of some {G}aussian processes}, Indiana Univ. Math. J.
  \textbf{20} (1970/71), 565--578.

\bibitem[Kal21]{Kalinin2021}
Alexander Kalinin, \emph{Support characterization for regular path-dependent
  stochastic {V}olterra integral equations}, Electron. J. Probab. \textbf{26}
  (2021), Paper No. 29, 29.

\bibitem[Kle14]{Klenke2014}
Achim Klenke, \emph{Probability theory}, second ed., Universitext, Springer,
  London, 2014, A comprehensive course.

\bibitem[Kru14]{Kruse2014}
Raphael Kruse, \emph{Strong and weak approximation of semilinear stochastic
  evolution equations}, Lecture Notes in Mathematics, vol. 2093, Springer,
  Cham, 2014.

\bibitem[KS91]{Karatzas1991}
Ioannis Karatzas and Steven~E. Shreve, \emph{Brownian motion and stochastic
  calculus}, second ed., Graduate Texts in Mathematics, vol. 113,
  Springer-Verlag, New York, 1991.

\bibitem[MS15]{Mytnik2015}
Leonid Mytnik and Thomas~S. Salisbury, \emph{{Uniqueness for Volterra-type
  stochastic integral equations}}, ArXiv preprint arXiv:1502.05513 (2015).

\bibitem[OHNO21]{Ouaddah}
A.~Ouaddah, J.~Henderson, J.~J. Nieto, and A.~Ouahab, \emph{A fractional bihari
  inequality and some applications to fractional differential equations and
  stochastic equations}, Mediterranean Journal of Mathematics \textbf{18}
  (2021).

\bibitem[{\O}Z93]{Oksendal1993}
Bernt {\O}ksendal and Tu~Sheng Zhang, \emph{The stochastic {V}olterra
  equation}, Barcelona {S}eminar on {S}tochastic {A}nalysis ({S}t. {F}eliu de
  {G}u\'{\i}xols, 1991), Progr. Probab., vol.~32, Birkh{\"a}user, Basel, 1993,
  pp.~168--202.

\bibitem[PP90]{Pardoux1990}
\'{E}tienne Pardoux and Philip Protter, \emph{Stochastic {V}olterra equations
  with anticipating coefficients}, Ann. Probab. \textbf{18} (1990), no.~4,
  1635--1655.

\bibitem[Pro85]{Protter1985}
Philip Protter, \emph{Volterra equations driven by semimartingales}, Ann.
  Probab. \textbf{13} (1985), no.~2, 519--530.

\bibitem[Pro92]{Pro1992}
\bysame, \emph{Stochastic integration and differential equation}, second ed.,
  Springer-Verlag, Berlin, Heidelberg, 1992.

\bibitem[RW00]{Rogers2000}
L.~C.~G. Rogers and David Williams, \emph{Diffusions, {M}arkov processes, and
  martingales. {V}ol. 2}, Cambridge Mathematical Library, Cambridge University
  Press, Cambridge, 2000, It\^{o} calculus, Reprint of the second (1994)
  edition.

\bibitem[RY99]{Revuz1999}
Daniel Revuz and Marc Yor, \emph{Continuous martingales and {B}rownian motion},
  third ed., Grundlehren der mathematischen Wissenschaften [Fundamental
  Principles of Mathematical Sciences], vol. 293, Springer-Verlag, Berlin,
  1999.

\bibitem[Ver12]{Veraar2012}
Mark Veraar, \emph{The stochastic {F}ubini theorem revisited}, Stochastics
  \textbf{84} (2012), no.~4, 543--551.

\bibitem[Wan08]{Wang2008}
Zhidong Wang, \emph{Existence and uniqueness of solutions to stochastic
  {V}olterra equations with singular kernels and non-{L}ipschitz coefficients},
  Statist. Probab. Lett. \textbf{78} (2008), no.~9, 1062--1071.

\bibitem[YW71]{Yamada1971}
Toshio Yamada and Shinzo Watanabe, \emph{On the uniqueness of solutions of
  stochastic differential equations}, J. Math. Kyoto Univ. \textbf{11} (1971),
  155--167.

\bibitem[Zha10]{Zhang2010}
Xicheng Zhang, \emph{Stochastic {V}olterra equations in {B}anach spaces and
  stochastic partial differential equation}, J. Funct. Anal. \textbf{258}
  (2010), no.~4, 1361--1425.

\end{thebibliography}
\bibliographystyle{amsalpha}

\end{document}